\documentclass[10pt,reqno]{amsart}
\usepackage{amssymb,amsmath,bm,geometry,graphics,url,color}
\usepackage{amsfonts}
\usepackage{graphicx}
\usepackage{epsfig,multirow}
\usepackage{amscd}
\usepackage{enumerate}
\usepackage{mathrsfs}
\usepackage{xypic}
\usepackage{stmaryrd}
\usepackage{fancybox}
\usepackage{exscale,array}
\usepackage{enumitem}%
\usepackage{color,cases}
\usepackage{hyperref}
\usepackage{algorithm}
\usepackage{algorithmic}
\usepackage{verbatim}
\usepackage{amsmath}
\usepackage{supertabular}
\usepackage{array}
\usepackage{multirow}
\usepackage{booktabs}
\usepackage{threeparttable}

\def\pdt2{\partial_t^2}
\def\pdx2{\partial_x^2}

\def\ii{\mathrm{i}}

\def\eps{\varepsilon}

\newcommand{\UU}{{\bf U}}
\newcommand{\VV}{{\bf V}}

\newcommand{\HH}{{\bf H}}
\newcommand{\EE}{{\bf E}}
\newcommand{\CU}{{\bf curl}}

\newcommand{\DD}{{\bf D}}

\newtheorem{theo}{Theorem}[section]
\newtheorem{lem}[theo]{Lemma}

\newtheorem{rem}[theo]{Remark}
\newtheorem{defi}[theo]{Definition}

\newtheorem{prop}[theo]{Proposition}
\numberwithin{equation}{section}

\newcommand{\norm}[1]{\left\Vert#1\right\Vert}

\newcommand{\abs}[1]{\left\vert#1\right\vert}
\title[Time  exponential integrator Fourier pseudospectral methods]{Time  exponential integrator Fourier pseudospectral methods with high accuracy and multiple conservation laws for three-dimensional Maxwell's equations}

\author[B. Wang]{Bin Wang}\address{\hspace*{-12pt}B.~Wang: School of Mathematics and Statistics, Xi'an Jiaotong University, 710049 Xi'an, China}
\email{wangbinmaths@xjtu.edu.cn}\urladdr{http://gr.xjtu.edu.cn/web/wangbinmaths/home}

\author[Y. L. Jiang]{Yaolin Jiang}
\address{\hspace*{-12pt}Corresponding author.  Y. L.~Jiang: School of Mathematics and Statistics, Xi'an Jiaotong University, 710049 Xi'an, China}
\email{yljiang@mail.xjtu.edu.cn}
\urladdr{http://gr.xjtu.edu.cn/web/yljiang}
\begin{document}
\maketitle
\dedicatory{}

\begin{abstract}
{Maxwell equations describe the propagation of
electromagnetic waves and are therefore fundamental to understanding
many problems encountered in the study of antennas and
electromagnetics.}
 The aim of this paper is to propose and {analyse} an efficient fully discrete scheme for solving three-dimensional Maxwell's equations.
  This is accomplished by combining time  exponential integrator and Fourier pseudospectral methods.
  Fast computation is implemented in the scheme by using the Fast Fourier Transform algorithm  which is well known in scientific computations.
  An optimal error estimate which is not encumbered by the CFL condition is established and the resulting scheme is proved to be of spectral accuracy
  in space and infinite-order accuracy in time. {Furthermore,} the  scheme is shown to have multiple conservation laws including discrete energy, helicity, momentum,
  symplecticity, and divergence-free field conservations. All the theoretical results of the accuracy and conservations are numerically illustrated by two numerical tests.
\end{abstract}

 \section{Introduction}
The Maxwell's equations  are the fundamental laws in electromagnetism and they describe the propagation and scattering of electromagnetic
waves.
They play a crucial role in a wide variety of applications in science and engineering such as wireless engineering, antennas, microwave circuits,   photonic cristals, radio-frequency,   aircraft radar, integrated optical circuits, waveguides, and interferometers.
This paper is devoted to the three-dimensional
 Maxwell's equations in an isotropic, homogeneous, and lossless medium which {are expressed in the following coupled form
 (see \cite{Leis86})}
 \begin{subequations}\label{maxsys}
\begin{numcases}%
\,\frac{\partial \HH}{\partial t} =-\frac{1}{\mu}\CU\  \EE,\ \ \ \frac{\partial \EE}{\partial t} =\frac{1}{\eps}\CU\  \HH,\ \ \ \
   \Omega\times[0,t_{end}],\qquad \textbf{(curl-equations)}\label{maxsys a}\\
\nabla\cdot (\eps \EE)= 0,\ \ \ \ \qquad \nabla\cdot (\eps \HH)= 0, \ \ \qquad
  \Omega\times [0,t_{end}],\qquad \textbf{(div-equations)}\label{maxsys b}
\end{numcases}
\end{subequations}
where  $\textbf{H}= (H_x ,H_y ,H_z )^{\intercal}: \Omega\times
[0,t_{end}] \rightarrow \mathbb{R}^3$ {stands for}
magnetic field intensity,
  $\textbf{E}= (E_x ,E_y ,E_z )^{\intercal}: \Omega\times [0,t_{end}]\rightarrow \mathbb{R}^3$ represents { electric field intensity},
 and  $\mu$ and $\eps$ denote  the magnetic permeability and electric permittivity, respectively.
 In this paper, the equations \eqref{maxsys} and their
initial values
\begin{equation}\label{IV}\HH_0(x,y,z)=\HH(x,y,z,0),\qquad { \EE_0(x,y,z)=\EE(x,y,z,0),}\end{equation}
are considered on the cuboid space domain $\Omega=[x_L,x_R]\times [y_L,y_R]\times [z_L,z_R]$ and periodic boundary conditions are required on the domain
$\partial \Omega\times [0,t_{end}].$ It is well known that Maxwell's equations \eqref{maxsys} have unique smooth solutions for all time
if the initial data  \eqref{IV} are suitably smooth \cite{Leis86}. The \textbf{div-equations} \eqref{maxsys b} can be derived from the \textbf{curl-equations} \eqref{maxsys a}
by taking divergence, and  they hold automatically  if $\HH_0$ and $\EE_0$ are divergence-free.

Due to the  great importance and diversity of applications,
Maxwell's equations have been researched for more than 150 years and
there has been a great interest in their numerical analysis in the
last couple of decades. Various numerical methods have been
investigated for the Maxwell's equations. The first kind of scheme
was the finite-difference time domain (FDTD) method which was
firstly proposed by { Yee  in \cite{3}},  and further developed and
{analysed} in \cite{Monk94,n38,n41,n43,n44,n49}. However,
these Yee-based FDTD methods require very small temporal step-size
to ensure their  stability. In order to improve the efficiency, the
alternating direction implicit (ADI) technique was proposed and some
unconditionally stable ADI-FDTD schemes were formulated  {
\cite{6,n25,n22,4,3,5}}. Besides finite-difference space
discretization,  other space discretization techniques have also
become  common such as { discontinuous Galerkin (dG) methods
\cite{Descombes13,Descombes16,n33,Moya12} or {the}
Fourier pseudo-spectral method \cite{n29}}. Concerning the schemes
of the time integration, some  time methods are applied to the
spatially {discretised} Maxwell's equations such as  the
explicit method \cite{Grote10},  Verlet method \cite{Fahs09},
Runge-Kutta (RK) methods { \cite{Hochbruck151,Hochbruck16}},
splitting methods \cite{Hochbruck15,Verwer11},  low-storage RK
schemes \cite{Diehl10}, multiscale methods
\cite{Hochbruck19,Henning16}, and exponential RK { methods
\cite{t13,yang}}.


It is noted that the above mentioned publications are devoted to the
accuracy and they do not pay attention to the structure
preservation. In recent years, due to the superior properties in
long time numerical computation over traditional numerical methods,
structure-preserving methods have been proved to be very powerful in
numerical simulations \cite{Lubich}.  For the Maxwell's equations
\eqref{maxsys}, they admit many  physical invariants: energy
conservation laws, symplectic  conservation laws, helicity
conservation laws, momentum conservation laws and divergence-free
fields. These invariants are very important in the long time
propagation of the electromagnetic waves \cite{n44}.

To present these physical invariants, we rewrite the  two \textbf{curl-equations} \eqref{maxsys a} as   a bi-Hamiltonian system  { (\cite{add5})} which  reads $
\frac{\partial}{\partial t}  \begin{pmatrix}
                                \HH   \\
                              \EE
                                 \end{pmatrix}
   = \left(
       \begin{array}{cc}
         \mathbf{0} & - \frac{\CU}{\eps \mu}\\
         \frac{\CU}{\eps \mu} & \mathbf{0} \\
       \end{array}
     \right) \begin{pmatrix}
                                \frac{\delta \mathcal{H}}{\delta  \HH}  \\
                                \frac{\delta \mathcal{H}}{\delta  \EE}
                                 \end{pmatrix}$,
where the quadratic Hamiltonian functional is
$ \mathcal{H}=\frac{1}{2}\int_{\Omega}\big(\eps\abs{\EE}^2+\mu\abs{\HH}^2 \big)dx dy dz$ with  the Euclidean norm $\abs{\cdot}$.
The solutions $\EE,\ \HH$  satisfy  the following energy conservation
laws (ECLs)  for $w = x,y$ or $z$ (\cite{n5,n14})
\begin{equation}\label{energyc}
  \begin{aligned}
& \frac{d }{d t} \mathcal{E}_1^{exact}(t)= 0,\quad \mathcal{E}_1^{exact}(t):=  \mathcal{H}= \frac{1}{2}\int_{\Omega}\big(\eps\abs{\EE}^2+\mu\abs{\HH}^2 \big)dx dy dz,\\
& \frac{d }{d t}\mathcal{E}_2^{exact}(t)= 0,\quad \mathcal{E}_2^{exact}(t):= \frac{1}{2} \int_{\Omega}\big(\eps\abs{\partial_t\EE}^2+\mu\abs{\partial_t \HH}^2 \big)dx dy dz,\\
& \frac{d }{d t} \mathcal{E}_3^{exact}(t)= 0,\quad \mathcal{E}_3^{exact}(t):= \frac{1}{2}\int_{\Omega}\big(\eps\abs{\partial_w\EE}^2+\mu\abs{\partial_w \HH}^2 \big)dx dy dz,\\
& \frac{d }{d t}\mathcal{E}_4^{exact}(t)= 0,\quad \mathcal{E}_4^{exact}(t):= \frac{1}{2} \int_{\Omega}\big(\eps\abs{\partial^2_{tw}\EE}^2+\mu\abs{\partial^2_{tw} \HH}^2 \big)dx dy dz.
  \end{aligned}
\end{equation}

With this Hamiltonian formulation, the Maxwell's equations \eqref{maxsys}
also satisfy the symplectic  conservation law  (\cite{n43})
\begin{equation}\label{symplecticityc}
  \begin{aligned}
& \frac{d }{d t} \int_{\Omega}\big(dE_x \wedge  dH_x + dE_y \wedge dH_y + dE_z \wedge dH_z  \big)dx dy dz= 0,
  \end{aligned}
\end{equation}
 the helicity conservation {laws} (\cite{n6})
\begin{equation}\label{helic}
  \begin{aligned}
 &\frac{d }{d t}\mathcal{H}_1^{exact}(t) = 0,\quad \mathcal{H}_1^{exact}(t):=  \int_{\Omega}\Big(\frac{ \EE ^{\intercal}(\CU\ \EE)}{2\mu} +
 \frac{ \HH ^{\intercal}(\CU\ \HH)}{2\eps}  \Big)dx dy dz,\\
 & \frac{d }{d t}\mathcal{H}_2^{exact}(t) = 0,\quad \mathcal{H}_2^{exact}(t):=   \int_{\Omega}\Big(\frac{(\partial_t \EE) ^{\intercal}(\CU\ \partial_t\EE)}{2\mu} +
 \frac{ (\partial_t\HH) ^{\intercal}(\CU\ \partial_t\HH)}{2\eps}  \Big)dx dy dz,
  \end{aligned}
\end{equation}
and  the momentum conservation laws  (\cite{n6})
\begin{equation}\label{momenc}
  \begin{aligned}
& \frac{d }{d t} \mathcal{M}_1^{exact}(t) = 0,\quad \mathcal{M}_1^{exact}(t):=  \int_{\Omega} (\HH ^{\intercal}  \partial_w \EE) dx dy dz, \quad \ \textmd{where}\ \ w=x,y,z,\\
 & \frac{d }{d t} \mathcal{M}_2^{exact}(t) = 0,\quad \mathcal{M}_2^{exact}(t):=  \int_{\Omega} (\EE ^{\intercal}  \partial_w \HH) dx dy dz, \quad \ \textmd{where}\ \ w=x,y,z.
  \end{aligned}
\end{equation}

These physical invariants are important for Maxwell's equations.   Naturally it is
desirable to propose a numerical scheme preserving them in the discrete sense (\cite{Lubich,WW21}). Thus, structure-preserving algorithms which can inherit these original
physical features as much as possible have gained remarkable success in the numerical analysis of
Maxwell's equations.  Concerning the structure-preserving algorithms of
Maxwell's equations, three categories have been received much attention in recent years: sympletic methods, divergence-free methods and energy-preserving methods.

There has been a great interest in solving Maxwell's equations by using symplectic  methods (see, e.g. \cite{n7,n21,n25,n40,n42,n43,n51}), which can  preserve the sympletic conservation law \eqref{symplecticityc} of the equations.
In order to make the numerical solution satisfy the  div-equations \eqref{maxsys b}, divergence-free methods were analysed (see, e.g. \cite{n15,n35}). Another important component of the structure-preserving
methods is the energy-preserving method. A kind of energy-conserved splitting method was proposed in \cite{n13} for two dimensional (2D) Maxwell's equations and in  \cite{n14} for the three dimensional (3D) case. Some energy preserving and unconditionally stable splitting schemes \cite{n5,Johannes19,n25,n22,n28} were further proposed.  However, most of these   energy-conserved
splitting schemes have only second order accuracy in both time and space at most.  In order to improve the accuracy, another energy-conserved scheme with fourth order accuracy in time was given in \cite{n6}.

There is no doubt that the idea to make use of structure-preserving
algorithms for Maxwell's equations {is} by no means new,
but there are still many key issues in such kind of algorithms which
remain to be well researched.
 In this paper, a kind of scheme with  high accuracy, low cost and multiple conservation laws is derived {and analised} for Maxwell's equations.
The main contributions of this paper are as follows:

a) Among all the existing structure-preserving methods for Maxwell's
equations, explicit methods usually suffer from step size
restrictions due to stability requirements (CFL condition) and
implicit methods can use larger time steps at the cost of solving
linear or even nonlinear systems. Meanwhile, all the methods have
limited accuracy (up to order four) in time. In this paper, we will
formulate a kind of explicit fully discrete scheme for
\eqref{maxsys} which is not  encumbered by the CFL condition (any
large time {step-size} is acceptable) and has spectral
accuracy in space and infinite-order accuracy in time.

b) Usually the structure-preserving algorithms of
Maxwell's equations can preserve some physical invariants. The new scheme proposed in this paper can simultaneously preserve the energy, symplecticity, helicity, momentum and divergence-free fields  conservation
laws.

c) For all the energy-preserving methods of Maxwell's equations,
they are implicit and the iterative procedure is needed. However,
the fully discrete scheme {proposed in} this paper is
explicit and can use any large time {step-size,} and thus
its computational cost is very low. Meanwhile, the scheme can be
implemented by using the matrix {diagonalisation} method
and Fast Fourier Transform algorithm which are very efficient in
scientific computing.

d) The proposed scheme needs the   regularity $
C^1(0,t_{end};[H^r_p(\Omega)]^3)$ with $r>3/2$ of $\HH$ and $\EE$ to
get the spectral accuracy in space and infinite-order accuracy in
time, which  is lower than those methods (spectral accuracy in space
and second or fourth-order accuracy in time) required {in
the publications (see, e.g. \cite{n5,n6,n7}).}

The rest of the paper is {organised} as follows. In
Section \ref{sec:2}, we present the formulation of the fully
discrete scheme and discuss its cost. The convergent analysis is
made in Section \ref{sec:3}.  Section \ref{sec:4}
{presents} the conservative properties of the proposed
scheme. Two numerical experiments are displayed in
 Section \ref{sec:5}  and the results demonstrate the high accuracy and exact conservation laws of the proposed scheme.

\section{Description of the fully discrete scheme}\label{sec:2}
   \subsection{Analytic framework}

 We begin this subsection with presenting some notations.  For a set $K\subset \Omega$  and the vectors fields
  $\UU, \widetilde{\UU}, \VV, \widetilde{\VV}:K\rightarrow \mathbb{R}^3$   the $L^2(K)$-inner product is denoted by
$\langle \UU,\widetilde{\UU}\rangle_{K}= \int_{K} \UU \cdot \widetilde{\UU} dx dy dz$ and for $F\subset \partial K$ we denote
$\langle \UU,\widetilde{\UU}\rangle_{F}= \int_{F} \UU|_F \cdot \widetilde{\UU}|_F d\sigma$.
Denoting by $\mathbf{u} = (\UU,\VV)$ and $\widehat{\mathbf{u}} = (\widehat{\UU},\widehat{\VV})$,  the weighted inner products are given by
$$\langle \UU,\widetilde{\UU}\rangle_{\alpha,K}=\langle \alpha\UU,\widetilde{\UU}\rangle_{K},\quad
\langle \mathbf{u},\widetilde{\mathbf{u}}\rangle_{\alpha\times \beta,K}=\langle \UU,\widetilde{\UU}\rangle_{\alpha,K}+
\langle \VV,\widetilde{\VV}\rangle_{\beta,K}$$
for the positive weight functions $\alpha,\beta:
\Omega\rightarrow \mathbb{R}^+ $. The corresponding norms are immediately obtained
as $\norm{\UU}^2_{\alpha}=\langle \UU,\UU\rangle_{\alpha,K}$ and $\norm{\mathbf{u}}_{\alpha\times \beta}^2=\norm{\UU}^2_{\alpha}+\norm{\VV}^2_{\beta}$.

   The  two \textbf{curl-equations} \eqref{maxsys a}  can be written as  an abstract Cauchy
problem
\begin{equation}\label{Cauchyform}
\frac{\partial}{\partial t}  \begin{pmatrix}
                               \sqrt{\mu} \HH   \\
                             \sqrt{\eps} \EE
                                 \end{pmatrix}
   =\mathcal{C} \begin{pmatrix}
                               \sqrt{\mu} \HH   \\
                             \sqrt{\eps} \EE
                                 \end{pmatrix},
\end{equation}
 where $\mathcal{C}= \left(
       \begin{array}{cc}
         \mathbf{0} & - \mathcal{C}_{\EE}\\
         \mathcal{C}_{\HH} & \mathbf{0} \\
       \end{array}
     \right)=\left(
       \begin{array}{cc}
         \mathbf{0} & - \frac{\CU}{\sqrt{\mu\eps}}\\
         \frac{\CU}{\sqrt{\mu\eps}} & \mathbf{0} \\
       \end{array}
     \right)$ is the Maxwell operator.   For the $\CU$ operator,
its graph space  is given by
$H(\CU,\Omega)=\{\UU\in L^2(\Omega)^3\mid \CU \UU\in L^2(\Omega)^3\},$
which is endowed   with the inner product
$\langle \UU,\VV\rangle_{H(\CU,\Omega)}=\langle \UU,\VV\rangle_{\Omega}+
\langle\CU  \UU,\CU  \VV\rangle_{\Omega}$
for all $\UU,\VV \in H(\CU,\Omega)$
and the associated norm given by $\norm{\UU}^2_{H(\CU,\Omega)}=\langle \UU,\UU\rangle_{H(\CU,\Omega)}$.
Denote  $H_0(\CU,\Omega)$  the closure of $C^{\infty}_0(\Omega)^3=\{v \in C^{\infty}(\Omega) | \textmd{supp}(v) \subset \Omega\ \textmd{is\ compact}\}^3$ with respect to the
norm $\norm{\cdot}_{H(\CU,\Omega)}$.
With these notations, it is well known that the Maxwell operator $\mathcal{C}$ is skew-adjoint w.r.t.
$\langle \cdot,\cdot \rangle_{\alpha\times \beta,\Omega}$ (\cite{Hochbruck15}), i.e.
$\langle \mathcal{C}_{\HH}\HH,\EE \rangle_{\eps,\Omega}=\langle\HH, \mathcal{C}_{\EE}\EE \rangle_{\mu,\Omega}$ for $ \HH\in D(\mathcal{C}_{\HH})$ and $\EE\in D(\mathcal{C}_{\EE}).$
Meanwhile, the  Maxwell operator $\mathcal{C}$ with domain
$D(\mathcal{C}) = D(\mathcal{C}_{\EE}) \times D(\mathcal{C}_{\HH} )  = H(\CU,\Omega) \times H_0(\CU,\Omega) $
  generates a unitary $C_0$-group $e^{t\mathcal{C}}$ (\cite{Hochbruck15}) { on a Hilbert space $X$}.

%

Based on the above results,  the well-posedness of Maxwell's equations can be derived by  Stone's theorem.
  If the initial value $(\HH_0, \EE_0)\in D(\mathcal{C})$, the two curl-equations \eqref{maxsys a}  have a unique solution
  in $C^1(0,t_{end};{ X}) \cap C (0,t_{end};D(\mathcal{C}))$ (\cite{Hochbruck15}) which can be
  given by
$ \begin{pmatrix}
                               \sqrt{\mu} \HH(t)   \\
                             \sqrt{\eps} \EE(t)
                                 \end{pmatrix}
   =e^{t\mathcal{C}} \begin{pmatrix}
                               \sqrt{\mu} \HH_0   \\
                             \sqrt{\eps} \EE_0
                                 \end{pmatrix}.$
                                According to this formula { and the
unitarity of  $e^{t\mathcal{C}}$, the solution is bounded
{by}
                                   \begin{equation}\label{bound solu} \norm{(\HH,\EE)}_{\mu\times \eps,\Omega}\leq \norm{(\HH_0,\EE_0)}_{\mu\times \eps,\Omega}.\end{equation}}
If these solutions $\HH,\EE$ are required to be smooth  enough and   the initial values  \eqref{IV} satisfy
the two \textbf{div-equations} \eqref{maxsys b}, i.e.
$\nabla\cdot (\eps \EE_0)= 0,\ \nabla\cdot (\eps \HH_0)= 0,$
then the two div-equations \eqref{maxsys b} hold true for $\HH,\EE$ at any $t\geq 0.$
In the rest parts of this section, the numerical scheme will be
derived for the abstract Cauchy problem \eqref{Cauchyform}.
{Then,} in the next section we will show the fact that
the proposed numerical solution  satisfies the  two
\textbf{div-equations} \eqref{maxsys b}.
 \subsection{Spatial {discretisation}}
 To achieve high order accuracy in treating the space,
the Fourier pseudo-spectral method is a very good
{discretisation. With this we are hopeful of obtaining}
high order accuracy and be implemented  by the fast Fourier
transform (FFT) algorithm \cite{n29,Shen,Trefethen}.

For the three-dimensional domain $\Omega=[x_L,x_R]\times
[y_L,y_R]\times [z_L,z_R]$, define a series of collocation points
$x_j =x_L + (j-1)h_x,\ y_k =y_L + (k-1)h_y,\ z_l =z_L + (l-1)h_z,$ $
j= 1,2,\ldots,N_x,\ k= 1,2,\ldots,N_y,\ l= 1,2,\ldots,N_z$, where
$h_x= (x_R-x_L )/N_x,\ h_y= (y_R-y_L )/N_y,\ h_z= (z_R-z_L )/N_z$
with even integers $N_x, N_y,$ and $N_z$. Denote the time
{step-size} by $\Delta t = t_{end}/N_t$ { for some
integer $N_t$}
  and let $t_n=n \Delta t$. The value of the function $E(x,y,z,t)$ at the
node $(x_j ,y_k ,z_l ,t_n )$ is denoted by $E^n_{j,k,l}$.

Denote a three-dimensional smooth function defined on $\Omega\times [0,t_{end}]$ by $U(x,y,z,t)$.
The interpolation space of this function is considered as
{ $$\mathcal{S}_{N_S}=\textmd{span}\{g_j(x)g_k(y)g_l(z),\ \ j= 1,2,\ldots,N_x,\ k= 1,2,\ldots,N_y,\ l= 1,2,\ldots,N_z\},$$}
where $N_S=N_xN_yN_z$ and $g_j(x), g_k(y), g_l(z)$ are trigonometric polynomials of degree $N_x/2,N_y/2,N_z/2$, given respectively by {
\begin{equation*}
  \begin{aligned}
& g_j(x)=\frac{1}{N_x}\sum\limits_{\abs{m}\leq N_x/2}^{'}e^{\ii m\nu_x(x-x_j)},\ \
 g_k(y)=\frac{1}{N_y}\sum\limits_{\abs{m}\leq N_y/2}^{'}e^{\ii m\nu_y(y-y_k)},\ \
 g_l(z)=\frac{1}{N_z}\sum\limits_{\abs{m}\leq N_z/2}^{'}e^{\ii m\nu_z(z-z_l)}.
  \end{aligned}
\end{equation*}}
Here { $\ii=\sqrt{-1}$ and} the prime indicates that
the first and last terms in the summation are taken with the factor
$1/2$, and $\nu_w=\frac{2\pi}{w_R-w_L}$ for $w=x,y,z.$
 Interpolating
$U(x,y,z,t)$ at collocation points $(x_j ,y_k ,z_l)$ gives $$
U(x,y,z,t)\approx  \mathcal{I}_{N_S}U(x,y,z,t)=\sum\limits_{j=1}^{N_x}\sum\limits_{k=1}^{N_y}\sum\limits_{l=1}^{N_z}
 U_{j,k,l}(t)g_j(x)g_k(y)g_l(z),$$
where $ U_{j,k,l}(t)= U(x_j ,y_k ,z_l,t)$ and its vector form is denoted by
$$\textbf{U}=(U_{1,1,1}, U_{2,1,1},\ldots, U_{N_x,1,1},U_{1,2,1}, U_{2,2,1},\ldots, U_{N_x,2,1},\ldots,
U_{1,N_y,N_z}, U_{2,N_y,N_z},\ldots, U_{N_x,N_y,N_z}) ^{\intercal}.$$

 Consider the
following   trigonometric polynomials as an approximation
for the solution of  \eqref{Cauchyform}
\begin{equation}\label{inter EH}
  \begin{aligned}
 \mathcal{I}_{N_S}\EE^{\mathcal{F}}= (\mathcal{I}_{N_S}E^{\mathcal{F}}_x, \mathcal{I}_{N_S}E^{\mathcal{F}}_y,\mathcal{I}_{N_S}E^{\mathcal{F}}_z) ^{\intercal},\ \ \
  \mathcal{I}_{N_S}\HH^{\mathcal{F}}= (\mathcal{I}_{N_S}H^{\mathcal{F}}_x, \mathcal{I}_{N_S}H^{\mathcal{F}}_y,\mathcal{I}_{N_S}H^{\mathcal{F}}_z) ^{\intercal},
  \end{aligned}
\end{equation}
which is required to satisfy $
\frac{\partial}{\partial t}  \begin{pmatrix}
                               \sqrt{\mu} \mathcal{I}_{N_S}\HH^{\mathcal{F}}   \\
                             \sqrt{\eps} \mathcal{I}_{N_S}\EE^{\mathcal{F}}
                                 \end{pmatrix}
   = \mathcal{C} \begin{pmatrix}
                               \sqrt{\mu} \mathcal{I}_{N_S}\HH^{\mathcal{F}}   \\
                             \sqrt{\eps} \mathcal{I}_{N_S}\EE^{\mathcal{F}}
                                 \end{pmatrix}.$
In order to calculate $
\CU  (\mathcal{I}_{N_S} \EE^{\mathcal{F}})
,$ we make partial differential with respect to $x$ and evaluate the
resulting expression at collocation points $(x_j ,y_k ,z_l)$:
{\begin{equation*}
  \begin{aligned}
\frac{\partial }{\partial x}\mathcal{I}_{N_S}E^{\mathcal{F}}_x (x_j ,y_k ,z_l)=&\sum\limits_{j'=1}^{N_x}\sum\limits_{k'=1}^{N_y}\sum\limits_{l'=1}^{N_z}
 E^{\mathcal{F}}_{x,j',k',l'}(t)\frac{d}{d
 x}g_{j'}(x_j)g_{k'}(y_k)g_{l'}(z_l)\\
 =&\Big[\big(I_{N_z} \otimes I_{N_y} \otimes D_x\big)\EE^{\mathcal{F}}_{x}\Big]_{N_xN_y(l-1)+N_x(k-1)+j},
  \end{aligned}
\end{equation*}}
where $\otimes$ is the Kronecker product, $I_{N_w}$ is the identity matrix of dimension $N_w \times N_w$,
$$\EE^{\mathcal{F}}_{x}=(E^{\mathcal{F}}_{x,1,1,1}, \ldots, E^{\mathcal{F}}_{x,N_x,1,1},E^{\mathcal{F}}_{x,1,2,1}, \ldots, E^{\mathcal{F}}_{x,N_x,2,1},\ldots,
E^{\mathcal{F}}_{x,1,N_y,N_z},  \ldots, E^{\mathcal{F}}_{x,N_x,N_y,N_z}) ^{\intercal},$$ and the spectral differential matrix is explicitly given by $(D_w)_{j,l}=  \frac{1}{2}\nu_w(-1)^{j+l}\cot(\nu_w (w_j-w_l)/2)$ for $ j\neq l$ and others elements are zero with $j,l=1,2,\ldots,N_w$ and $w=x,y,z.$
Similarly, one has
$
 \frac{\partial }{\partial y}\mathcal{I}_{N_S}E^{\mathcal{F}}_x (x_j ,y_k ,z_l)
 =\Big[\big(I_{N_z} \otimes D_y \otimes I_{N_x}\big)\EE^{\mathcal{F}}_{x}\Big]_{N_xN_y(l-1)+N_x(k-1)+j}$ and $
  \frac{\partial }{\partial z}\mathcal{I}_{N_S}E^{\mathcal{F}}_x (x_j ,y_k ,z_l)
 =\Big[\big(D_z  \otimes I_{N_y} \otimes I_{N_x} \big)\EE^{\mathcal{F}}_{x}\Big]_{N_xN_y(l-1)+N_x(k-1)+j}.
$

Based on the above results and with the notations
$\widetilde{\EE}  = (\EE^{\mathcal{F}}_{x}, \EE^{\mathcal{F}}_y,\EE^{\mathcal{F}}_z) ^{\intercal},\
   \widetilde{\HH}= (\HH^{\mathcal{F}}_x, \HH^{\mathcal{F}}_y,\HH^{\mathcal{F}}_z) ^{\intercal},$
  the following ordinary differential equations are obtained
\begin{equation}\label{maxreformFD}
\frac{d}{d t}  \begin{pmatrix}
                               \sqrt{\mu}\widetilde{\HH}   \\
                             \sqrt{\eps}\widetilde{\EE}
                                 \end{pmatrix}
   = \frac{1}{\sqrt{\mu\eps}}\left(
       \begin{array}{cc}
         \mathbf{0} & - \DD\\
         \DD & \mathbf{0} \\
       \end{array}
     \right) \begin{pmatrix}
                               \sqrt{\mu} \widetilde{\HH}   \\
                             \sqrt{\eps}\widetilde{\EE}
                                 \end{pmatrix},
\end{equation}
where
$$\DD=\left(
        \begin{array}{ccc}
          \textbf{0} & -D_z  \otimes I_{N_y} \otimes I_{N_x} & I_{N_z}  \otimes  D_y  \otimes  I_{N_x}\\
         D_z  \otimes I_{N_y} \otimes I_{N_x} & \textbf{0} & -I_{N_z}  \otimes I_{N_y} \otimes  D_x\\
          -I_{N_z}  \otimes  D_y  \otimes  I_{N_x}  &I_{N_z}  \otimes I_{N_y} \otimes  D_x & \textbf{0} \\
        \end{array}
      \right).
$$
The initial values of \eqref{maxreformFD} are obtained by  considering the initial values \eqref{IV} at   the  collocation points.

 \subsection{Time {discretisation}}
 In this part, we formulate the time {discretisation} of the system  \eqref{maxreformFD} by using exponential integrators. This kind of method has been well researched and
 has been made successful {applications}  in  many systems {(see, e.g. \cite{LubichW,Hochbruck10,WZ})}. However, exponential integrators suffer from the fact that  direct
 {discretisation} of
the three-dimensional Maxwell's equations  on a grid requires the
storage and computation of a { $(6N_xN_yN_z)\times (6N_xN_yN_z)$}
matrix exponential, which is prohibitive from a computational point
of view in many cases.
 In order to make the computation of the matrix exponential  be achievable and effective,  matrix {diagonalisation} method and vector-valued trigonometric functions are employed in this section.

Denote $\mathcal{F}_{N_w}$
  the matrix of DFT coefficients with entries given by $(\mathcal{F}_{N_w})_{j,k}=(e^{2\pi { \ii}/N_w})^{-jk}$ and
$(\mathcal{F}^{-1}_{N_w})_{j,k}=\frac{1}{N_w}(e^{2\pi { \ii}/N_w})^{jk}$. Then it is known that $\DD=\mathcal{F}^{-1}\Lambda \mathcal{F}$
with
\begin{equation*}
  \begin{aligned}
 & \mathcal{F}=\textmd{diag}\Big(
          \mathcal{F}_{N_z}  \otimes \mathcal{F}_{N_y} \otimes \mathcal{F}_{N_x},\mathcal{F}_{N_z}  \otimes \mathcal{F}_{N_y} \otimes \mathcal{F}_{N_x} ,\mathcal{F}_{N_z}  \otimes \mathcal{F}_{N_y} \otimes \mathcal{F}_{N_x}\Big),\\
&\Lambda=\left(
        \begin{array}{ccc}
          \textbf{0} & -\Lambda_z  \otimes I_{N_y} \otimes I_{N_x} & I_{N_z}  \otimes  \Lambda_y  \otimes  I_{N_x}\\
         \Lambda_z  \otimes I_{N_y} \otimes I_{N_x} & \textbf{0} & -I_{N_z}  \otimes I_{N_y} \otimes \Lambda_x\\
          -I_{N_z}  \otimes \Lambda_y  \otimes  I_{N_x}  &I_{N_z}  \otimes I_{N_y} \otimes \Lambda_x & \textbf{0} \\
        \end{array}
      \right),\\
      &\Lambda_w={ \ii}\Omega_w:={ \ii}\nu_w\textmd{diag}\Big(0,1,\ldots,\frac{N_w}{2}-1,0,{ -\frac{N_w}{2}+1},\ldots,-2,-1\Big)\ \ \ \textmd{for }\ w=x,y,z.
  \end{aligned}
\end{equation*}
Then letting
$
 \widehat{\EE}= \mathcal{F} \widetilde{\EE},\   \widehat{\HH}= \mathcal{F} \widetilde{\HH},$
 \eqref{maxreformFD} can be reformulated as
 \begin{equation}\label{maxreformFD1}
\frac{d}{d t}  \begin{pmatrix}
                               \sqrt{\mu} \widehat{\HH}   \\
                             \sqrt{\eps} \widehat{\EE}
                                 \end{pmatrix}
   =  \widehat{\Lambda} \begin{pmatrix}
                               \sqrt{\mu}  \widehat{\HH}   \\
                             \sqrt{\eps} \widehat{\EE}
                                 \end{pmatrix},
\end{equation}
{where} $\widehat{\Lambda}=\frac{1}{\sqrt{\mu\eps}}\left(
       \begin{array}{cc}
         \mathbf{0} & - \Lambda\\
         \Lambda & \mathbf{0} \\
       \end{array}
     \right)$.
 The exact solution of  \eqref{maxreformFD1} is
$
  \begin{pmatrix}
                               \sqrt{\mu}\widehat{\HH}(t)   \\
                             \sqrt{\eps}\widehat{\EE}(t)
                                 \end{pmatrix}
   = e^{t\widehat{\Lambda}}  \begin{pmatrix}
                               \sqrt{\mu} \widehat{\HH}(0)   \\
                             \sqrt{\eps}\widehat{\EE}(0)
                                 \end{pmatrix}.$
Here  $e^{t\widehat{\Lambda}}$ is the exponential of the matrix $t\widehat{\Lambda}$. This exponential itself is again a $6N_S\times 6N_S$
matrix, or in other words a linear operator from { $\mathbb{C}^{6N_S}$ to $\mathbb{C}^{6N_S}$}. Even more, it is easy to see that
$e^{t\widehat{\Lambda}}$ with $t\geq 0$ is a strongly continuous semigroup.  It is clear that $\widehat{\Lambda}$ is a skew-hermitian matrix. Thus the matrix exponential $e^{t\widehat{\Lambda}}$  is unitary and thus
satisfies { $\norm{e^{t\widehat{\Lambda}} }=1$}.

Now the key step is to compute  the  matrix exponential  $e^{t\widehat{\Lambda}}$. The results are given by the following two propositions.
\begin{prop}\label{pro1}
For the  matrix exponential, one has  $e^{t\widehat{\Lambda}}=\left(
                \begin{array}{cc}
                  \cos\big(\frac{t}{\sqrt{\mu\eps}}\Lambda\big) & -\sin\big(\frac{t}{\sqrt{\mu\eps}}\Lambda\big) \\
                  \sin\big(\frac{t}{\sqrt{\mu\eps}}\Lambda\big) & \cos\big(\frac{t}{\sqrt{\mu\eps}}\Lambda\big) \\
                \end{array}
              \right).$
\end{prop}
\begin{proof}
 It is easy to see that $\widehat{\Lambda}^k=\frac{1}{(\sqrt{\mu\eps})^k}\left(
       \begin{array}{cc}
         \mathbf{0} & - \Lambda^k\\
         \Lambda^k & \mathbf{0} \\
       \end{array}
     \right)$ for $k=4m+1$, $\frac{1}{(\sqrt{\mu\eps})^k}\left(
       \begin{array}{cc}
         - \Lambda^k &   \mathbf{0} \\
         \mathbf{0} & - \Lambda^k \\
       \end{array}
     \right)$ for $k=4m+2$, $\frac{1}{(\sqrt{\mu\eps})^k}\left(
       \begin{array}{cc}
         \mathbf{0} & \Lambda^k\\
         -\Lambda^k & \mathbf{0} \\
       \end{array}
     \right)$ for $k=4m+3$, and $\frac{1}{(\sqrt{\mu\eps})^k}\left(
       \begin{array}{cc}
         \Lambda^k &   \mathbf{0} \\
         \mathbf{0} & \Lambda^k \\
       \end{array}
     \right)$ for $k=4m$.
     %
%
Then one gets \begin{equation*}
  \begin{aligned} e^{t\widehat{\Lambda}}=&\sum_{k=0}^{\infty}\frac{1}{k!}\widehat{\Lambda}^k
  =\sum_{m=0}^{\infty} \left(
       \begin{array}{cc}
      \frac{ \Lambda^{4m}}{(4m)! (\sqrt{\mu\eps})^{4m}}-\frac{ \Lambda^{4m+2}}{(4m+2)! (\sqrt{\mu\eps})^{4m+2}}   &   -\frac{ \Lambda^{4m+1}}{(4m+1)! (\sqrt{\mu\eps})^{4m+1}} +\frac{ \Lambda^{4m+3}}{(4m+3)! (\sqrt{\mu\eps})^{4m+3}}   \\
        \frac{ \Lambda^{4m+1}}{(4m+1)! (\sqrt{\mu\eps})^{4m+1}} -\frac{ \Lambda^{4m+3}}{(4m+3)! (\sqrt{\mu\eps})^{4m+3}}  & \frac{ \Lambda^{4m}}{(4m)! (\sqrt{\mu\eps})^{4m}}-\frac{ \Lambda^{4m+2}}{(4m+2)! (\sqrt{\mu\eps})^{4m+2}} \\
       \end{array}
     \right),
  \end{aligned} \end{equation*}
  which {confirms} the result. \hfill
 \end{proof}

 \begin{prop} \label{pro2}
It is deduced that
\begin{equation*}
  \begin{aligned}&\cos(t\Lambda/\sqrt{\mu\eps})=\left(
                                                   \begin{array}{ccc}
                                                     \textbf{c}_{11}(t\Lambda/\sqrt{\mu\eps}) & \textbf{c}_{12} (t\Lambda/\sqrt{\mu\eps}) & \textbf{c}_{13}(t\Lambda/\sqrt{\mu\eps})  \\
                                                    \textbf{c}_{12}(t\Lambda/\sqrt{\mu\eps})  & \textbf{c}_{22}(t\Lambda/\sqrt{\mu\eps})  & \textbf{c}_{23} (t\Lambda/\sqrt{\mu\eps}) \\
                                                    \textbf{c}_{13} (t\Lambda/\sqrt{\mu\eps}) & \textbf{c}_{23}(t\Lambda/\sqrt{\mu\eps})  & \textbf{c}_{33}(t\Lambda/\sqrt{\mu\eps})  \\
                                                   \end{array}
                                                 \right),\\ &\sin(t\Lambda/\sqrt{\mu\eps})=\left(
                                                   \begin{array}{ccc}
                                                     \textbf{0}  & -\textbf{s}_{12}(t\Lambda/\sqrt{\mu\eps})  & \textbf{s}_{13}(t\Lambda/\sqrt{\mu\eps})  \\
                                                    \textbf{s}_{12}(t\Lambda/\sqrt{\mu\eps})  & \textbf{0} &- \textbf{s}_{23}(t\Lambda/\sqrt{\mu\eps})  \\
                                                    -\textbf{s}_{13} (t\Lambda/\sqrt{\mu\eps}) & \textbf{s}_{23}(t\Lambda/\sqrt{\mu\eps})  & \textbf{0} \\
                                                   \end{array}
                                                 \right),
  \end{aligned} \end{equation*}
where
\begin{equation}\label{matrixs}
  \begin{aligned}
 &\textbf{c}_{11}(t\Lambda)=I+\frac{t^2(\Omega_3^2+\Omega_2^2)(\cosh(\sqrt{-\Psi(t\Omega)})-I)}{\Psi(t\Omega)}, \quad \textbf{c}_{12}(t\Lambda)=-\frac{t^2\Omega_2 \Omega_1 (\cosh(\sqrt{-\Psi(t\Omega)})-I)}{\Psi(t\Omega)},\\
 & \textbf{c}_{22}(t\Lambda)=I+\frac{t^2(\Omega_3^2+\Omega_1^2)(\cosh(\sqrt{-\Psi(t\Omega)})-I)}{\Psi(t\Omega)}, \quad  \textbf{c}_{13}(t\Lambda)=-\frac{t^2\Omega_3 \Omega_1 (\cosh(\sqrt{-\Psi(t\Omega)})-I)}{\Psi(t\Omega)},\\
  & \textbf{c}_{33}(t\Lambda)=I+\frac{t^2(\Omega_2^2+\Omega_1^2)(\cosh(\sqrt{-\Psi(t\Omega)})-I)}{\Psi(t\Omega)},  \quad  \textbf{c}_{23}(t\Lambda)=-\frac{t^2\Omega_3 \Omega_2 (\cosh(\sqrt{-\Psi(t\Omega)})-I)}{\Psi(t\Omega)},\\
   &\textbf{s}_{12}(t\Lambda)=\frac{{ \ii}t\Omega_3 \sinh(\sqrt{-\Psi(t\Omega)})}{\sqrt{-\Psi(t\Omega)}},\ \
  \textbf{s}_{13}(t\Lambda)=\frac{{ \ii}t\Omega_2 \sinh(\sqrt{-\Psi(t\Omega)})}{\sqrt{-\Psi(t\Omega)}},\ \
   \textbf{s}_{23}(t\Lambda)=\frac{{ \ii}t\Omega_1 \sinh(\sqrt{-\Psi(t\Omega)})}{\sqrt{-\Psi(t\Omega)}},
  \end{aligned}
\end{equation}
with $\Omega_1=I_{N_z}  \otimes I_{N_y} \otimes \Omega_x$,
$\Omega_2=I_{N_z}  \otimes  \Omega_y  \otimes  I_{N_x}$,
$\Omega_3=\Omega_z  \otimes I_{N_y} \otimes I_{N_x}$, and $\Psi(t\Omega)=t^2(\Omega_1^2+\Omega_2^2+\Omega_3^2).$
\end{prop}
\begin{proof}The proof is similar to that of Proposition  \ref{pro1} and we skip it for { brevity. \hfill}
\end{proof}

 \subsection{Fully discrete scheme}
We now present the novel fully discrete scheme  of   the three-dimensional Maxwell's equations \eqref{maxsys}.
\begin{defi}\label{dIUA-PE}(\textbf{Fully discrete scheme.})
For solving the three-dimensional Maxwell's equations \eqref{maxsys},  the fully discrete scheme  is defined as follows.
\begin{description}
  \item[Step 1. (Space {step-sizes} and collocation points)] Choose even integers $N_x, N_y, N_z$ to get the space stepsizes $h_x= (x_R-x_L )/N_x,\
h_y= (y_R-y_L )/N_y,\ h_z= (z_R-z_L )/N_z$ and collocation points
$(x_j ,y_k ,z_l)$.
  \item[Step 2. (Initial values)]  Considering the initial values \eqref{IV} at   these collocation points   gives $(\EE^0_{x},\EE^0_{y},\EE^0_{z})  $ and $(\HH^0_{x},\HH^0_{y},\HH^0_{z}) ,$
{for} $w=x,y$ or $z$
\begin{equation}\label{iiv}
  \begin{aligned}\EE^0_{w}=\big(&E_{w}(x_1,y_1,z_1,0), \ldots, E_{w}(x_{N_x},y_1,z_1,0), E_{w}(x_1,y_2,z_1,0), \ldots, \\& E_{w}(x_{N_x},y_2,z_1,0),  \ldots,
E_{w}(x_1,y_{N_y},z_{N_z},0),  \ldots, E_{w}(x_{N_x},y_{N_y},z_{N_z},0)\big) ^{\intercal}, \\
 \HH^0_{w}=\big(&H_{w}(x_1,y_1,z_1,0), \ldots, H_{w}(x_{N_x},y_1,z_1,0), H_{w}(x_1,y_2,z_1,0), \ldots, \\& H_{w}(x_{N_x},y_2,z_1,0),\ldots,
H_{w}(x_1,y_{N_y},z_{N_z},0),  \ldots, H_{w}(x_{N_x},y_{N_y},z_{N_z},0)\big) ^{\intercal}.\end{aligned}
\end{equation}
Then the  initial values for the system  \eqref{maxreformFD1} are given by
 \begin{equation}\label{tran0}%
 \widehat{\EE}_w^0=\big(\mathcal{F}_{N_z}  \otimes \mathcal{F}_{N_y} \otimes \mathcal{F}_{N_x}\big) \EE^0_{w},\ \ \ \  \widehat{\HH}_w^0=\big(\mathcal{F}_{N_z}  \otimes \mathcal{F}_{N_y} \otimes \mathcal{F}_{N_x}\big) \HH^0_{w}.
\end{equation}

 \item[Step 3. (Time integration)]  For solving the three-dimensional Maxwell's equations \eqref{maxsys} on the time interval $[0,t_{end}]$,
  the  following scheme is considered
  \begin{equation*}
\begin{array}[c]{ll}%
 \sqrt{\mu}\widehat{\HH}^{t_{end}}_x= &\sqrt{\mu}\big(\textbf{c}_{11}(t_{end}\Lambda/\sqrt{\mu\eps}) \widehat{\HH}^{0}_x
+\textbf{c}_{12}(t_{end}\Lambda/\sqrt{\mu\eps}) \widehat{\HH}^{0}_y+\textbf{c}_{13}(t_{end}\Lambda/\sqrt{\mu\eps}) \widehat{\HH}^{0}_z\big)\\
&-\sqrt{\eps}\big(-\textbf{s}_{12}(t_{end}\Lambda/\sqrt{\mu\eps}) \widehat{\EE}^{0}_y
+\textbf{s}_{13}(t_{end}\Lambda/\sqrt{\mu\eps}) \widehat{\EE}^{0}_z\big),\\
 \sqrt{\mu}\widehat{\HH}^{t_{end}}_y= &\sqrt{\mu}\big(\textbf{c}_{12}(t_{end}\Lambda/\sqrt{\mu\eps}) \widehat{\HH}^{0}_x
+\textbf{c}_{22}(t_{end}\Lambda/\sqrt{\mu\eps}) \widehat{\HH}^{0}_y+\textbf{c}_{23}(t_{end}\Lambda/\sqrt{\mu\eps}) \widehat{\HH}^{0}_z\big)\\
&-\sqrt{\eps}\big(\textbf{s}_{12}(t_{end}\Lambda/\sqrt{\mu\eps}) \widehat{\EE}^{0}_x
-\textbf{s}_{23}(t_{end}\Lambda/\sqrt{\mu\eps}) \widehat{\EE}^{0}_z\big),\\
 \sqrt{\mu}\widehat{\HH}^{t_{end}}_z= &\sqrt{\mu}\big(\textbf{c}_{13}(t_{end}\Lambda/\sqrt{\mu\eps}) \widehat{\HH}^{0}_x
+\textbf{c}_{23}(t_{end}\Lambda/\sqrt{\mu\eps}) \widehat{\HH}^{0}_y+\textbf{c}_{33}(t_{end}\Lambda/\sqrt{\mu\eps}) \widehat{\HH}^{0}_z\big)\\
&-\sqrt{\eps}\big(-\textbf{s}_{13}(t_{end}\Lambda/\sqrt{\mu\eps}) \widehat{\EE}^{0}_x
+\textbf{s}_{23}(t_{end}\Lambda/\sqrt{\mu\eps}) \widehat{\EE}^{0}_y\big),\\
 \sqrt{\eps}\widehat{\EE}^{t_{end}}_x= &\sqrt{\eps}\big(\textbf{c}_{11}(t_{end}\Lambda/\sqrt{\mu\eps}) \widehat{\EE}^{0}_x
+\textbf{c}_{12}(t_{end}\Lambda/\sqrt{\mu\eps}) \widehat{\EE}^{0}_y+\textbf{c}_{13}(t_{end}\Lambda/\sqrt{\mu\eps}) \widehat{\EE}^{0}_z\big)\\
&+\sqrt{\mu}\big(-\textbf{s}_{12}(t_{end}\Lambda/\sqrt{\mu\eps}) \widehat{\HH}^{0}_y
+\textbf{s}_{13}(t_{end}\Lambda/\sqrt{\mu\eps}) \widehat{\HH}^{0}_z\big),\\
 \sqrt{\eps}\widehat{\EE}^{t_{end}}_y= &\sqrt{\eps}\big(\textbf{c}_{12}(t_{end}\Lambda/\sqrt{\mu\eps}) \widehat{\EE}^{0}_x
+\textbf{c}_{22}(t_{end}\Lambda/\sqrt{\mu\eps}) \widehat{\EE}^{0}_y+\textbf{c}_{23}(t_{end}\Lambda/\sqrt{\mu\eps}) \widehat{\EE}^{0}_z\big)\\
&+\sqrt{\mu}\big(\textbf{s}_{12}(t_{end}\Lambda/\sqrt{\mu\eps}) \widehat{\HH}^{0}_x
-\textbf{s}_{23}(t_{end}\Lambda/\sqrt{\mu\eps}) \widehat{\HH}^{0}_z\big),\\
 \sqrt{\eps}\widehat{\EE}^{t_{end}}_z= &\sqrt{\eps}\big(\textbf{c}_{13}(t_{end}\Lambda/\sqrt{\mu\eps}) \widehat{\EE}^{0}_x
+\textbf{c}_{23}(t_{end}\Lambda/\sqrt{\mu\eps}) \widehat{\EE}^{0}_y+\textbf{c}_{33}(t_{end}\Lambda/\sqrt{\mu\eps}) \widehat{\EE}^{0}_z\big)\\
&+\sqrt{\mu}\big(-\textbf{s}_{13}(t_{end}\Lambda/\sqrt{\mu\eps}) \widehat{\HH}^{0}_x
+\textbf{s}_{23}(t_{end}\Lambda/\sqrt{\mu\eps}) \widehat{\HH}^{0}_y\big),\\
\end{array}\end{equation*}
where $\textbf{c}_{\cdot}$ and $\textbf{s}_{\cdot}$ are determined
{by} \eqref{matrixs}.

 \item[Step 4. (Final result)]   The final results $$\EE^{t_{end}}_w\approx (E_w(x_j,y_k,z_l,t_{end}))_{j,k,l}\ \  \textmd{and} \ \ \HH^{t_{end}}_w\approx (H_w(x_j,y_k,z_l,t_{end}))_{j,k,l}$$
approximating the solution of \eqref{maxsys} at the collocation points $(x_j ,y_k ,z_l)$ and at the time $t_{end}$ are given by
 \begin{equation}\label{tran1}%
 \EE^{t_{end}}_w=\big(\mathcal{F}^{-1}_{N_z}  \otimes \mathcal{F}^{-1}_{N_y} \otimes \mathcal{F}^{-1}_{N_x}\big)\widehat{\EE}^{t_{end}}_w, \quad \HH^{t_{end}}_w=\big(\mathcal{F}^{-1}_{N_z}  \otimes \mathcal{F}^{-1}_{N_y} \otimes \mathcal{F}^{-1}_{N_x}\big)\widehat{\HH}^{t_{end}}_w.
\end{equation}
\end{description}
\end{defi}

 \subsection{Fast computation and its cost}

In this part, we discuss the complexity of the proposed scheme. A fast solver
can be used  to increase computational efficiency. The idea is based on the diagonal matrix
  and the Fast Fourier Transform (FFT) algorithm.

We first discuss the complexity of deriving initial values.
Computing the  collocation points  requires $\mathcal{O}(N_x+N_y+N_z)$ arithmetic operations and   storage.
Then the storage cost and computational cost of the
  values $\EE^0_{w}, \HH^0_{w}$ \eqref{iiv}  are both $\mathcal{O}(N_x N_y N_z)$.
The Fast Fourier Transform (FFT)  algorithm can be applied to obtain $ \widehat{\EE}_w^0$,  $ \widehat{\HH}_w^0$  \eqref{tran0}
and the storage cost is  $\mathcal{O}(N_x N_y N_z)$.
The details of this  procedure and its cost are presented in Algorithm 1.

{We then} discuss the computational characteristics of
the proposed algorithm. We have to compute the coefficients
\eqref{matrixs} and {  this can be achieved by vector
operations since $\Lambda$ is diagonal.} Thus the cost and storage
of this step can be reduced to  $\mathcal{O} (N_xN_yN_z)$  from
$\mathcal{O} (N_x^2N_y^2N_z^2)$. {Moreover,} it is noted
here that a great advantage of the scheme is that arbitrary large
time {step-size} is accepted. The final results at any
$t_{end}$ are obtained from the initial values by  only one step
computation with a time step $ \triangle t=t_{end}$. Therefore, the
cost  of the scheme is very low in comparison  with the standard
methods using a time {step-size:} $0<\triangle t<1$. The
detailed complexity of the fully discrete scheme  \ref{dIUA-PE} is
{stated} in Algorithm 2.



 \begin{algorithm}[t!]\label{al1}
\caption{{ \textbf{(Initial values)}} The goal of the algorithm is to obtain the initial values $\widehat{\EE}_w^0$ and  $ \widehat{\HH}_w^0$ for the method proposed in this paper. }
 \begin{description}
             \item[Input] $N_x, N_y, N_z$ (even integers)
             \item[Output] $ \widehat{\EE}_w^0$,  $ \widehat{\HH}_w^0$ (initial values for our scheme)
\end{description}
\begin{algorithmic}[1]
\STATE {Compute $h_x= (x_R-x_L )/N_x,\
h_y= (y_R-y_L )/N_y,\ h_z= (z_R-z_L )/N_z$. \\
\textbf{Cost: $\mathcal{O}(1)$. Storage: $\mathcal{O}(1)$.}}
\STATE {Compute $x_j =x_L + (j-1)h_x,\ y_k =y_L + (k-1)h_y,\ z_l =z_L + (l-1)h_z,$ $
j= 1,2,\ldots,N_x,\ k= 1,2,\ldots,N_y,\ l= 1,2,\ldots,N_z$.\\
\textbf{Cost: $\mathcal{O}(N_x+N_y+N_z)$. Storage: $\mathcal{O}(N_x+N_y+N_z)$.}}
\STATE {Compute the values $\EE^0_{w}, \HH^0_{w}$ \eqref{iiv} from \eqref{IV} such that $$[\EE^{0}_w]_{N_xN_y(l-1)+N_x(k-1)+j}= E_w(x_j,y_k,z_l,0),\ \   [\HH^{0}_w]_{N_xN_y(l-1)+N_x(k-1)+j}= H_w(x_j,y_k,z_l,0)$$ for $w=x,y,z $ and $j= 1,2,\ldots,N_x,\ k= 1,2,\ldots,N_y,\ l= 1,2,\ldots,N_z$.\\
\textbf{Cost: $\mathcal{O}\big(N_xN_yN_z\big)$. Storage: $\mathcal{O}(N_xN_yN_z)$.}}
\STATE {By  Fast Fourier Transform (FFT), compute the initial values  $$\widehat{\EE}_w^0=\big(\mathcal{F}_{N_z}  \otimes \mathcal{F}_{N_y} \otimes \mathcal{F}_{N_x}\big) \EE^0_{w},\ \ \ \  \widehat{\HH}_w^0=\big(\mathcal{F}_{N_z}  \otimes \mathcal{F}_{N_y} \otimes \mathcal{F}_{N_x}\big) \HH^0_{w}.$$
\textbf{Cost: $\mathcal{O}(\textmd{FFT})$. Storage: $\mathcal{O}(N_xN_yN_z)$.}}
\end{algorithmic}
\end{algorithm}

 \begin{algorithm}[t!]\label{al123}
\caption{\textbf{(Fully discrete scheme)} The goal of the algorithm is to obtain numerical solution  $\EE^{t_{end}}_w\in \mathbb{C}^{N_xN_yN_z}$ and $\HH^{t_{end}}_w\in \mathbb{C}^{N_xN_yN_z}$
such that $[\EE^{t_{end}}_w]_{N_xN_y(l-1)+N_x(k-1)+j}\approx E_w(x_j,y_k,z_l,t_{end})$
and $[\HH^{t_{end}}_w]_{N_xN_y(l-1)+N_x(k-1)+j}\approx H_w(x_j,y_k,z_l,t_{end})$ for $w=x,y,z.$}
 \begin{description}
             \item[Input] $N_x, N_y, N_z,\widehat{\EE}_w^0,\ \widehat{\HH}_w^0$ (obtained by Algorithm 1) and $t_{end}$
             \item[Output] $\EE_w^{t_{end}}$,  $\HH_w^{t_{end}}$ (such that $[\EE^{t_{end}}_w]_{N_xN_y(l-1)+N_x(k-1)+j}\approx E_w(x_j,y_k,z_l,t_{end})$
and $[\HH^{t_{end}}_w]_{N_xN_y(l-1)+N_x(k-1)+j}\approx H_w(x_j,y_k,z_l,t_{end})$)
\end{description}
\begin{algorithmic}[1]
\STATE {Set $\mathbf{a}_w:=\frac{2\pi}{w_R-w_L}\Big(0,1,\ldots,\frac{N_w}{2}-1,0,\frac{N_w}{2}+1,\ldots,-2,-1\Big)^{\intercal}\in \mathbb{R}^{N_w}\  \textmd{for }\ w=x,y,z.$\\
\textbf{Cost: $\mathcal{O}(N_x+N_y+N_z)$. Storage: $\mathcal{O}(N_x+N_y+N_z)$.}}
\STATE {Compute $\mathbf{b}_x=\mathbf{1}_{N_z}  \otimes \mathbf{1}_{N_y} \otimes \mathbf{a}_x$,
$\mathbf{b}_y=\mathbf{1}_{N_z}  \otimes  \mathbf{a}_y  \otimes  \mathbf{1}_{N_x}$,
$\mathbf{b}_z=\mathbf{a}_z  \otimes \mathbf{1}_{N_y} \otimes \mathbf{1}_{N_x}$, and $\Psi =\kappa^2(\mathbf{b}_x^2+\mathbf{b}_y^2+\mathbf{b}_z^2).$ Here $\kappa=t_{end}/\sqrt{\mu\eps}$ and ${ \mathbf{1}_{N_w}=(1,1,\ldots,1)^{\intercal}\in \mathbb{R}^{N_w}}.$\\
\textbf{Cost: $\mathcal{O} (N_xN_yN_z)$.  Storage: $\mathcal{O} (N_xN_yN_z)$.}}
\STATE { Compute  (\textbf{only once})
  \begin{equation*}
  \begin{aligned}
   &\textbf{r}_1:= (\cosh(\sqrt{-\Psi })-\mathbf{1}_{N_xN_yN_z})./\Psi , \ \textbf{r}_2:=\sinh(\sqrt{-\Psi })./\sqrt{-\Psi },\\
 &\textbf{c}_{11} =\mathbf{1}_{N_xN_yN_z}+
 \kappa^2(\mathbf{b}_z^2+\mathbf{b}_y^2).*\textbf{r}_1 , \ \
   \textbf{c}_{12} =- \kappa^2\mathbf{b}_y.* \mathbf{b}_x .* \textbf{r}_1 ,\ \ \textbf{s}_{12} = {\ii}\kappa\mathbf{b}_z.* \textbf{r}_2,\\
 & \textbf{c}_{22} =\mathbf{1}_{N_xN_yN_z}+ \kappa^2(\mathbf{b}_z^2+
 \mathbf{b}_x^2).*\textbf{r}_1 , \ \
   \textbf{c}_{13} =- \kappa^2\mathbf{b}_z .*\mathbf{b}_x .*\textbf{r}_1 ,\ \ \textbf{s}_{13} ={\ii} \kappa\mathbf{b}_y .*\textbf{r}_2,\\
  & \textbf{c}_{33} =\mathbf{1}_{N_xN_yN_z}+\kappa^2(\mathbf{b}_y^2
  +\mathbf{b}_x^2).*\textbf{r}_1 , \ \
  \textbf{c}_{23} =- \kappa^2\mathbf{b}_z .*\mathbf{b}_y .*\textbf{r}_1 ,\ \ \textbf{s}_{23} ={\ii}\kappa\mathbf{b}_x.* { \textbf{r}_2,}
  \end{aligned}
\end{equation*}
{ where $./$ and $.*$ denote the element-by-element division and multiplication of two vectors, respectively.}
\\
\textbf{Cost: $\mathcal{O} (N_xN_yN_z)$.  Storage: $\mathcal{O} (N_xN_yN_z)$.}}
 \STATE {Compute  (\textbf{only one step}) \begin{equation*}
\begin{array}[c]{ll}%
&\widehat{\HH}^{0}_x:=\sqrt{\mu}\widehat{\HH}^{0}_x,  \ \widehat{\HH}^{0}_y:= \sqrt{\mu}\widehat{\HH}^{0}_y,\ \widehat{\HH}^{0}_z:= \sqrt{\mu}\widehat{\HH}^{0}_z,\  \widehat{\EE}^{0}_x:= \sqrt{\eps}\widehat{\EE}^{0}_x,\   \widehat{\EE}^{0}_y:= \sqrt{\eps}\widehat{\EE}^{0}_y,\ \widehat{\EE}^{0}_z:= \sqrt{\eps}\widehat{\EE}^{0}_z, \\
&\widehat{\HH}^{t_{end}}_x=\textbf{c}_{11}  .*\widehat{\HH}^{0}_x
+\textbf{c}_{12}  .*\widehat{\HH}^{0}_y+\textbf{c}_{13}  .*\widehat{\HH}^{0}_z +\textbf{s}_{12} .* \widehat{\EE}^{0}_y
-\textbf{s}_{13}  .*\widehat{\EE}^{0}_z,\\
&\widehat{\HH}^{t_{end}}_y=\textbf{c}_{12}  .*\widehat{\HH}^{0}_x
+\textbf{c}_{22}  .*\widehat{\HH}^{0}_y+\textbf{c}_{23}  .*\widehat{\HH}^{0}_z
 -\textbf{s}_{12}  .*\widehat{\EE}^{0}_x
+\textbf{s}_{23} .* \widehat{\EE}^{0}_z,\\
&\widehat{\HH}^{t_{end}}_z=\textbf{c}_{13}  .*\widehat{\HH}^{0}_x
+\textbf{c}_{23}  .*\widehat{\HH}^{0}_y+\textbf{c}_{33}  .*\widehat{\HH}^{0}_z
 +\textbf{s}_{13}  .*\widehat{\EE}^{0}_x
-\textbf{s}_{23} .* \widehat{\EE}^{0}_y,\\
 & \widehat{\EE}^{t_{end}}_x=\textbf{c}_{11}  .*\widehat{\EE}^{0}_x
+\textbf{c}_{12}  .*\widehat{\EE}^{0}_y+\textbf{c}_{13} .* \widehat{\EE}^{0}_z -\textbf{s}_{12}  .*\widehat{\HH}^{0}_y
+\textbf{s}_{13}  .*\widehat{\HH}^{0}_z,\\
 &\widehat{\EE}^{t_{end}}_y= \textbf{c}_{12}  .*\widehat{\EE}^{0}_x
+\textbf{c}_{22}  .*\widehat{\EE}^{0}_y+\textbf{c}_{23} .* \widehat{\EE}^{0}_z +\textbf{s}_{12} .* \widehat{\HH}^{0}_x
-\textbf{s}_{23} .* \widehat{\HH}^{0}_z,\\
& \widehat{\EE}^{t_{end}}_z= \textbf{c}_{13} .* \widehat{\EE}^{0}_x
+\textbf{c}_{23}  .*\widehat{\EE}^{0}_y+\textbf{c}_{33} .* \widehat{\EE}^{0}_z-\textbf{s}_{13} .* \widehat{\HH}^{0}_x
+\textbf{s}_{23}  .*\widehat{\HH}^{0}_y,\\
&\widehat{\HH}^{t_{end}}_x:=\frac{1}{\sqrt{\mu}}\widehat{\HH}^{t_{end}}_x,  \ \widehat{\HH}^{t_{end}}_y:= \frac{1}{\sqrt{\mu}}\widehat{\HH}^{t_{end}}_y,\ \widehat{\HH}^{t_{end}}_z:= \frac{1}{\sqrt{\mu}}\widehat{\HH}^{t_{end}}_z,\\
& \widehat{\EE}^{t_{end}}_x:=\frac{1}{\sqrt{\eps}}\widehat{\EE}^{t_{end}}_x,\ \ \ \widehat{\EE}^{t_{end}}_y:= \frac{1}{\sqrt{\eps}}\widehat{\EE}^{t_{end}}_y,\ \ {\widehat{\EE}^{t_{end}}_z:= \frac{1}{\sqrt{\eps}}\widehat{\EE}^{t_{end}}_z.}
\end{array}\end{equation*}\\
\textbf{Cost: $\mathcal{O} (N_xN_yN_z)$. Storage: $\mathcal{O} (N_xN_yN_z)$.}}
 \STATE {Using Inverse Fast Fourier Transform (IFFT), compute for $w=x,y,z$
  $$\EE^{t_{end}}_w=\big(\mathcal{F}^{-1}_{N_z}  \otimes \mathcal{F}^{-1}_{N_y} \otimes \mathcal{F}^{-1}_{N_x}\big)\widehat{\EE}^{t_{end}}_w, \quad {\HH^{t_{end}}_w=\big(\mathcal{F}^{-1}_{N_z}  \otimes \mathcal{F}^{-1}_{N_y} \otimes \mathcal{F}^{-1}_{N_x}\big)\widehat{\HH}^{t_{end}}_w.}$$
\textbf{Cost: $\mathcal{O}(\textmd{IFFT})$. Storage: $\mathcal{O} (N_xN_yN_z)$.}}
\end{algorithmic}
\end{algorithm}


\section{Convergence}\label{sec:3}
In this proof, we focus on the error estimates of the proposed
scheme. For simplicity we consider the cubic domain $\Omega=
[0,2\pi]^3$ with the spatial grid points $N_x = N_y = N_z =N.$ A
general cuboid domain can be linearly mapped into $\Omega=
[0,2\pi]^3$. Let $C^{\infty}_p (\Omega)$ be the set of infinitely
differentiable periodic functions with period $2\pi$, and $H^{r}_p
(\Omega)$ be the closure of $C^{\infty}_p (\Omega)$ in $H^{r}
(\Omega)$. Define the inner product by $\langle
u,v\rangle_{\Omega}=\frac{1}{8\pi^3}\int_{\Omega}u(x,y,z)
\overline{v(x,y,z)} dx dy dz$ and the discrete inner product and
norm by, respectively, $\langle
u,v\rangle_{N}=\frac{1}{N^3}\sum\limits_{j=1}^{N}\sum\limits_{k=1}^{N}\sum\limits_{l=1}^{N}u(x_j,y_k,z_l)
\overline{v(x_j,y_k,z_l)},\ \norm{u}^2_N=\langle u,u\rangle_{N}.$
The norm and seminorm of $H^{r} (\Omega)$ are denoted by
$\norm{\cdot}_r$ and  $\abs{\cdot}_r$, respectively. In particular,
$\norm{\cdot}_0=\norm{\cdot}$. Let the interpolation space
$\mathcal{S}^I_N=\Big\{u|u=\sum\limits_{{
\abs{j},\abs{k},\abs{l}}\leq\frac{N}{2}}
\frac{\hat{u}_{j,k,l}}{c_jc_kc_l}e^{\ii(jx+ky+lz)}:
\overline{\hat{u}_{j,k,l}}=\hat{u}_{-j,-k,-l},
\hat{u}_{\frac{N}{2},k,l}=\hat{u}_{-\frac{N}{2},k,l},\hat{u}_{j,\frac{N}{2},l}=\hat{u}_{j,-\frac{N}{2},l},
\hat{u}_{j,k,\frac{N}{2}}=\hat{u}_{j,k,-\frac{N}{2}}\Big\}, $ where
$c_l=1$ for $\abs{l}< \frac{N}{2}$ and $c_l=2$ for $\abs{l}=
\frac{N}{2}$. We denote
$\mathcal{S}^O_N=\Big\{u|u=\sum\limits_{\abs{j},\abs{k},\abs{l}\leq\frac{N}{2}}
\hat{u}_{j,k,l}e^{\ii(jx+ky+lz)}:
\overline{\hat{u}_{j,k,l}}=\hat{u}_{-j,-k,-l} \Big\}.$
{We here remark} that $\mathcal{S}^I_N\subseteq
\mathcal{S}^O_N$ . We denote $\mathcal{P}^O_N: [L^2
(\Omega)]^3\rightarrow  [\mathcal{S}^O_N ]^3$ as the orthogonal
projection operator and recall the interpolation operator
$\mathcal{P}^I_N: [C (\Omega)]^3\rightarrow  [\mathcal{S}^I_N ]^3.$

\begin{lem}  (\cite{n5})
 For all $\mathbf{u}\in [\mathcal{S}^I_N]^3$, we have $\norm{\mathbf{u}}_0\leq \norm{\mathbf{u}}_N \leq  2\sqrt{2}\norm{\mathbf{u}}_0.$
If $\mathbf{u},\mathbf{v} \in [\mathcal{S}^I_N]^3$, then $\langle \partial_w \mathbf{u},\mathbf{v}\rangle_{N}=-\langle \mathbf{u},\partial_w  \mathbf{v}\rangle_{N}$ for $w=x,y,z.$
\end{lem}

\begin{lem}  (\cite{n5})
  $\langle \mathcal{P}^O_N \mathbf{u},\mathbf{v}\rangle_{N}=\langle \mathbf{u},\mathbf{v}\rangle_{N}$ for $ \mathbf{v}\in [\mathcal{S}^O_N ]^3.$  By noting  $\mathcal{P}^O_N\partial_w u=\partial_w \mathcal{P}^O_N u$ for $w=x,y,z$, we can see that   $\CU$ and $\mathcal{P}^O_N$
satisfy the
commutative law.
\end{lem}

\begin{lem} {  (\cite{add3})}
If $0\leq \alpha\leq r$ and $\mathbf{u}\in [H_p^{r}
(\Omega)]^3$, then $ \norm{\mathcal{P}^O_N \mathbf{u}-\mathbf{u}}_{ \alpha}\leq C N^{ \alpha-r}\abs{\mathbf{u}}_r$  and in addition if $r>3/2$ then  $ \norm{\mathcal{P}^I_N \mathbf{u}-\mathbf{u}}_{ \alpha}\leq C N^{ \alpha-r}\abs{\mathbf{u}}_r$.
\end{lem}


\begin{theo} \label{con thm}\textbf{(Convergence.)}
Suppose that the exact solution    $\HH, \EE\in C^1(0,t_{end};[H^r_p(\Omega)]^3)$ and the  initial values  $\HH_0, \EE_0\in [H^r_p(\Omega)]^3$, where $r>3/2$ and the initial values are assumed to be bounded. Let {$\EE^{t_{end}}  = (\EE^{t_{end}}_{x}, \EE^{t_{end}}_y,\EE^{t_{end}}_z) ^{\intercal},\
   \HH^{t_{end}}= (\HH^{t_{end}}_x, \HH^{t_{end}}_y,\HH^{t_{end}}_z) ^{\intercal}$} be the solutions of the { scheme \ref{dIUA-PE}.}
Then, for any fixed $t_{end}$ there exists a positive constant $C$ independent of { $\Delta t, h_x,h_y,h_z, t_{end},N,\mu,\eps$  such that
\begin{equation*}
  \begin{aligned}& \Big( \mu \norm{\HH^{t_{end}}-\HH(t_{end})}^2_N+ \eps \norm{\EE^{t_{end}}-\EE(t_{end})}^2_N\Big)^{\frac{1}{2}}\leq C (\sqrt{\mu}+\sqrt{\eps}) N^{-r},\end{aligned}
\end{equation*}}
where $N=N_x=N_y=N_z$.

\end{theo}
\begin{proof}
Let $\EE^{*}=\mathcal{P}^O_{N-2}\EE,\
\HH^{*}=\mathcal{P}^O_{N-2}\HH.$ The {projections} of
Eqs. \eqref{Cauchyform} are written as
\begin{equation}\label{maxreformP} \frac{\partial}{\partial t}
\begin{pmatrix}
                               \sqrt{\mu} \HH^{*}   \\
                             \sqrt{\eps} \EE^{*}
                                 \end{pmatrix}
   = \frac{1}{\sqrt{\mu\eps}}\left(
       \begin{array}{cc}
         \mathbf{0} & - \CU\\
         \CU & \mathbf{0} \\
       \end{array}
     \right) \begin{pmatrix}
                               \sqrt{\mu} \HH^{*}   \\
                             \sqrt{\eps} \EE^{*}
                                 \end{pmatrix}.
\end{equation}
Noting that   $ \EE^{*}\in[\mathcal{S}^O_{N-2} ]^3\subseteq
[\mathcal{S}^I_N]^3 \subseteq [\mathcal{S}^O_N]^3,$ we
{obtain} that
\begin{equation*}
  \begin{aligned}
\frac{\partial }{\partial x}E^{*}_w (x_j ,y_k ,z_l)=&\frac{\partial }{\partial x}\mathcal{P}^I_{N}E^{*}_w (x_j ,y_k ,z_l)=\Big[\DD_1\EE^{*}_{w}\Big]_{N_xN_y(l-1)+N_x(k-1)+j},\\
\frac{\partial }{\partial y}E^{*}_w (x_j ,y_k ,z_l)=&\frac{\partial }{\partial y}\mathcal{P}^I_{N}E^{*}_w (x_j ,y_k ,z_l)=\Big[\DD_2\EE^{*}_{w}\Big]_{N_xN_y(l-1)+N_x(k-1)+j},\\
\frac{\partial }{\partial z}E^{*}_w (x_j ,y_k ,z_l)=&\frac{\partial }{\partial z}\mathcal{P}^I_{N}E^{*}_w (x_j ,y_k ,z_l)=\Big[\DD_3\EE^{*}_{w}\Big]_{N_xN_y(l-1)+N_x(k-1)+j},\\
  \end{aligned}
\end{equation*}
where {
$\DD_1= I_{N_z}  \otimes I_{N_y} \otimes  D_x,\
\DD_2= I_{N_z}  \otimes  D_y  \otimes  I_{N_x},\
\DD_3=D_z  \otimes I_{N_y} \otimes I_{N_x},
$ and}
\begin{equation*}
  \begin{aligned}\EE^{*}_{w}=(&E^{*}_{w,1,1,1}, E^{*}_{w,2,1,1},\ldots, E^{*}_{w,N_x,1,1},E^{*}_{w,1,2,1}, E^{*}_{w,2,2,1},\ldots, E^{*}_{w,N_x,2,1},\\
&\ldots,
E^{*}_{w,1,N_y,N_z}, E^{*}_{w,2,N_y,N_z},\ldots, E^{*}_{w,N_x,N_y,N_z}) ^{\intercal}\ \ \ \textmd{for}\ \ w=x,y,z.  \end{aligned}
\end{equation*}
Similar results are obvious for $ \HH^{*}$.
Thus Eqs. \eqref{maxreformP}  are transformed into
\begin{equation}\label{maxreformPP}
\frac{d}{d t}  \begin{pmatrix}
                               \sqrt{\mu} \widetilde{\HH}^{*}   \\
                             \sqrt{\eps} \widetilde{\EE}^{*}
                                 \end{pmatrix}
   =\mathcal{D}  \begin{pmatrix}
                               \sqrt{\mu} \widetilde{\HH}^{*}   \\
                             \sqrt{\eps} \widetilde{\EE}^{*}
                                 \end{pmatrix},
\end{equation}
with $\mathcal{D}=\frac{1}{\sqrt{\mu\eps}}\left(
       \begin{array}{cc}
         \mathbf{0} & - \DD\\
         \DD & \mathbf{0} \\
       \end{array}
     \right)$ and
$  \widetilde{\EE}^{*}  = (\EE^{*}_{x}, \EE^{*}_y,\EE^{*}_z) ^{\intercal},\
  {  \widetilde{\HH}^{*}}= (\HH^{*}_x, \HH^{*}_y,\HH^{*}_z) ^{\intercal}.$

   On the other hand,  the Fully discrete scheme \eqref{dIUA-PE} is equivalent to finding the numerical solution
  $(\widetilde{\HH}, \widetilde{\EE})^{\intercal} \in [\mathcal{S}^I_N]^6$ such that
   \begin{equation}\label{proo 1}
\Big\langle \frac{d}{d t}  \begin{pmatrix}
                               \sqrt{\mu}\widetilde{\HH}   \\
                             \sqrt{\eps}\widetilde{\EE}
                                 \end{pmatrix},\left(
                                                 \begin{array}{c}
                                                   \vec{\mu} \\
                                                   \vec{\nu} \\
                                                 \end{array}
                                               \right)
                                 \Big\rangle_{N}
   = \Big\langle \mathcal{D}\begin{pmatrix}
                               \sqrt{\mu} \widetilde{\HH}   \\
                             \sqrt{\eps}\widetilde{\EE}
                                 \end{pmatrix},\left(
                                                 \begin{array}{c}
                                                   \vec{\mu} \\
                                                   \vec{\nu} \\
                                                 \end{array}
                                               \right)
                                 \Big\rangle_{N}
\end{equation}
for all  $(  \vec{\mu},  \vec{\nu})^{\intercal} \in [\mathcal{S}^I_N]^6$.

 Denote the errors $\mathcal{H}^t= \widetilde{\HH}(t)-   {  \widetilde{\HH}^{*}}(t)$ and $\mathcal{E}^t =\widetilde{\EE}(t)- {  \widetilde{\EE}^{*}}(t)$.
 Based on the formulae \eqref{maxreformPP}-\eqref{proo 1}, it is {clear} that
   {  \begin{equation*}
\Big\langle \frac{d}{d t}  \begin{pmatrix}
                               \sqrt{\mu}\mathcal{H}^t   \\
                             \sqrt{\eps}\mathcal{E}^t
                                 \end{pmatrix},\left(
                                                 \begin{array}{c}
                                                   \vec{\mu} \\
                                                   \vec{\nu} \\
                                                 \end{array}
                                               \right)
                                 \Big\rangle_{N}
   = \Big\langle\mathcal{D}\begin{pmatrix}
                               \sqrt{\mu} \mathcal{H}^t  \\
                             \sqrt{\eps} \mathcal{E}^t
                                 \end{pmatrix},\left(
                                                 \begin{array}{c}
                                                   \vec{\mu} \\
                                                   \vec{\nu} \\
                                                 \end{array}
                                               \right)
                                 \Big\rangle_{N},
\end{equation*}}
 which is
    {  \begin{equation*}
\Big\langle   \begin{pmatrix}
                               \sqrt{\mu}\mathcal{H}^t   \\
                             \sqrt{\eps}\mathcal{E}^t
                                 \end{pmatrix},\left(
                                                 \begin{array}{c}
                                                   \vec{\mu} \\
                                                   \vec{\nu} \\
                                                 \end{array}
                                               \right)
                                 \Big\rangle_{N}
   = \Big\langle e^{\mathcal{D} }\begin{pmatrix}
                               \sqrt{\mu} \mathcal{H}^0  \\
                             \sqrt{\eps} \mathcal{E}^0
                                 \end{pmatrix},\left(
                                                 \begin{array}{c}
                                                   \vec{\mu} \\
                                                   \vec{\nu} \\
                                                 \end{array}
                                               \right)
                                 \Big\rangle_{N}.
\end{equation*}}
 Taking $\left(
                                                 \begin{array}{c}
                                                   \vec{\mu} \\
                                                   \vec{\nu} \\
                                                 \end{array}
                                               \right)=\begin{pmatrix}
                               \sqrt{\mu}\mathcal{H}^t   \\
                             \sqrt{\eps}\mathcal{E}^t
                                 \end{pmatrix}+e^{\mathcal{D}}\begin{pmatrix}
                               \sqrt{\mu} \mathcal{H}^0  \\
                             \sqrt{\eps} \mathcal{E}^0
                                 \end{pmatrix}$
                                 leads to
          \begin{equation*}
  \begin{aligned} 0=&\norm{\begin{pmatrix}
                               \sqrt{\mu}\mathcal{H}^t   \\
                             \sqrt{\eps}\mathcal{E}^t
                                 \end{pmatrix}}_N^2-\norm{e^{\mathcal{D}}\begin{pmatrix}
                               \sqrt{\mu} \mathcal{H}^0  \\
                             \sqrt{\eps} \mathcal{E}^0
                                 \end{pmatrix}}_N^2+\Big\langle   \begin{pmatrix}
                               \sqrt{\mu}\mathcal{H}^t   \\
                             \sqrt{\eps}\mathcal{E}^t
                                 \end{pmatrix},e^{\mathcal{D} }\begin{pmatrix}
                               \sqrt{\mu} \mathcal{H}^0  \\
                             \sqrt{\eps} \mathcal{E}^0
                                 \end{pmatrix}
                                 \Big\rangle_{N}\\&-\Big\langle e^{\mathcal{D} }\begin{pmatrix}
                               \sqrt{\mu} \mathcal{H}^0  \\
                             \sqrt{\eps} \mathcal{E}^0
                                 \end{pmatrix},  \begin{pmatrix}
                               \sqrt{\mu}\mathcal{H}^t   \\
                             \sqrt{\eps}\mathcal{E}^t
                                 \end{pmatrix}
                                 \Big\rangle_{N}
                           = \norm{\begin{pmatrix}
                               \sqrt{\mu}\mathcal{H}^t   \\
                             \sqrt{\eps}\mathcal{E}^t
                                 \end{pmatrix}}_N^2-\norm{\begin{pmatrix}
                               \sqrt{\mu} \mathcal{H}^0  \\
                             \sqrt{\eps} \mathcal{E}^0
                                 \end{pmatrix}}_N^2     .\end{aligned}
\end{equation*}
This shows that $\mu \norm{ \mathcal{H}^t}_N^2+\eps \norm{ \mathcal{E}^t}_N^2=\mu \norm{ \mathcal{H}^0}_N^2+\eps \norm{ \mathcal{E}^0}_N^2.$

In what follows, we estimate $\norm{ \mathcal{H}^0}_N^2$   and      $\norm{ \mathcal{E}^0}_N^2$.
For $ \mathcal{H}^0=\widetilde{\HH}(0)- \mathcal{P}^O_{N-2}\HH(0)\in [\mathcal{S}^I_N]^3$, we transform the norm by using the result
$\norm{\mathcal{H}^0}_0\leq \norm{\mathcal{H}^0}_N \leq  2\sqrt{2}\norm{\mathcal{H}^0}_0$ and then study the bound of $\norm{\mathcal{H}^0}_0$.
To this end, we compute
       \begin{equation*}
  \begin{aligned}\norm{\mathcal{H}^0}_0=&\norm{\widetilde{\HH}(0)- \mathcal{P}^O_{N-2}\HH(0)}_0=
  \norm{ \mathcal{P}^I_N(0)\HH(0)- \mathcal{P}^O_{N-2}\HH(0)}_0\\\leq&
    \norm{ \mathcal{P}^I_N(0)\HH(0)-\HH(0)}_0+  \norm{\mathcal{P}^O_{N-2}\HH(0)-\HH(0)}_0\leq C N^{-r}.\end{aligned}
\end{equation*}
Similar result $\norm{\mathcal{E}^0}_N \leq  2\sqrt{2}\norm{\mathcal{E}^0}_0 \leq C N^{-r}$ can be obtained. Therefore,
{$\Big(\mu \norm{ \mathcal{H}^t}_N^2+\eps \norm{ \mathcal{E}^t}_N^2\Big)^{\frac{1}{2}}  \leq    C (\sqrt{\mu}+\sqrt{\eps}) N^{-r}.$}

  { With these estimates, we} are now {in a position} to present the error
  { \begin{equation*}
  \begin{aligned}& \Big(\mu \norm{\HH^{t_{end}}-\HH(t_{end})}^2_N+ \eps\norm{\EE^{t_{end}}-\EE(t_{end})}^2_N\Big)^{\frac{1}{2}}\\
  \leq&
   \Big(\mu \norm{\widetilde{\HH}(t_{end})-\HH(t_{end})}^2_N+ \eps\norm{\widetilde{\EE}(t_{end})-\EE(t_{end})}^2_N\Big)^{\frac{1}{2}}\\
    \leq &  \Big(\mu \norm{ \mathcal{H}^{t_{end}}+\mathcal{P}^O_{N-2}\HH(t_{end}) -\HH(t_{end})}^2_N+ \eps \norm{\mathcal{E}^{t_{end}}+\mathcal{P}^O_{N-2}\EE(t_{end}) -\EE(t_{end})}^2_N\Big)^{\frac{1}{2}}\\
        \leq & \sqrt{\mu} \norm{ \mathcal{H}^{t_{end}}+\mathcal{P}^O_{N-2}\HH(t_{end}) -\HH(t_{end})}_N+ \sqrt{\eps} \norm{\mathcal{E}^{t_{end}}+\mathcal{P}^O_{N-2}\EE(t_{end}) -\EE(t_{end})}_N\\
       \leq &    \sqrt{\mu}\norm{ \mathcal{P}^O_{N-2}\HH(t_{end}) -\HH(t_{end})}_N+ \sqrt{\eps}\norm{\mathcal{P}^O_{N-2}\EE(t_{end}) -\EE(t_{end})}_N + C(\sqrt{\mu}+\sqrt{\eps}) N^{-r} \\
       \leq &  C(\sqrt{\mu}+\sqrt{\eps}) N^{-r},\end{aligned}
\end{equation*}}
where we have used the boundedness \eqref{bound solu} of the solution to get
   {\begin{equation*}
  \begin{aligned}&  \sqrt{\mu}\norm{ \mathcal{P}^O_{N-2}\HH(t_{end}) -\HH(t_{end})}_N+ \sqrt{\eps}\norm{\mathcal{P}^O_{N-2}\EE(t_{end}) -\EE(t_{end})}_N\\ \leq&  2\sqrt{2}\Big(\sqrt{\mu}\norm{ \mathcal{P}^O_{N-2}\HH(t_{end}) -\HH(t_{end})}_0+ \sqrt{\eps}\norm{\mathcal{P}^O_{N-2}\EE(t_{end}) -\EE(t_{end})}_0\Big) \\
  \leq& C(\sqrt{\mu}+\sqrt{\eps})   (N-2)^{-r}.\end{aligned}
\end{equation*}}\hfill
\end{proof}
\begin{rem}
It is noted that  the scheme is of spectral accuracy in space and infinite-order accuracy in time. For the
  Maxwell's equations  with enough smoothness solutions,   the  scheme will converge with infinite-order accuracy both in space and in time.
\end{rem}
 \section{Structure preserving laws}\label{sec:4}
 In this section, we rigorously prove the discrete
structure preserving laws of the proposed scheme including the energy, helicity, momentum, symplecticity, and divergence-free field conservation laws.

\begin{theo} \label{ene law thm} 
\textbf{(Energy conservation laws.)}
The solutions { $\EE^{t_{end}},  \HH^{t_{end}}$} produced by the fully discrete  scheme \ref{dIUA-PE} satisfy the
discrete energy conservation laws
\begin{equation*}
  \begin{aligned}& \mathcal{E}_1^{t_{end}}=\mathcal{E}_1^{0},\ \ \mathcal{E}_2^{t_{end}}=\mathcal{E}_2^{0},\ \ \mathcal{E}_3^{t_{end}}=\mathcal{E}_3^{0},\ \ \mathcal{E}_4^{t_{end}}=\mathcal{E}_4^{0},  \ \ \mathcal{E}_5^{t_{end}}=\mathcal{E}_5^{0},\
  \ \mathcal{E}_6^{t_{end}}=\mathcal{E}_6^{0},\end{aligned}
\end{equation*}
where
\begin{equation*}
  \begin{aligned}\mathcal{E}_1^{t}=&\frac{\mu}{2}  \langle\HH^{t},  \HH^{t}\rangle_N+\frac{\eps}{2}  \langle\EE^{t},  \EE^{t}\rangle_N,\ \qquad
   \mathcal{E}_2^{t}=\frac{\mu}{2}  \Big\langle \frac{d}{dt} \HH^{t},  \frac{d}{dt} \HH^{t}\Big\rangle_N+\frac{\eps}{2} \Big\langle\frac{d}{dt} \EE^{t},  \frac{d}{dt} \EE^{t}\Big\rangle_N,\\
      \mathcal{E}_3^{t}=&\frac{\mu}{2} \sum_{w=x,y,z}\langle\DD_k\HH_w^{t},\DD_k \HH_w^{t}\rangle_N+
\frac{\eps}{2} \sum_{w=x,y,z} \langle\DD_k\EE_w^{t},\DD_k \EE_w^{t}\rangle_N\ \ \ \textmd{for}\ \ \ k=1,2,3,\\
   \mathcal{E}_4^{t}=&\frac{\mu}{2} \sum_{w=x,y,z}\langle\DD_k\frac{d}{dt}\HH_w^{t},\DD_k \frac{d}{dt}\HH_w^{t}\rangle_N+
\frac{\eps}{2} \sum_{w=x,y,z} \langle\DD_k\frac{d}{dt}\EE_w^{t},\DD_k \frac{d}{dt}\EE_w^{t}\rangle_N\ \ \ \textmd{for}\ \ \ k=1,2,3,\\
    \mathcal{E}_5^{t}=& \frac{\mu}{2} \sum_{w=x,y,z}\langle\HH_w^{t},\DD_k \HH_w^{t}\rangle_N+
\frac{\eps}{2} \sum_{w=x,y,z} \langle\EE_w^{t},\DD_k \EE_w^{t}\rangle_N\ \ \ \textmd{for}\ \ \ k=1,2,3,\\
     \mathcal{E}_6^{t}=& \frac{\mu}{2}\sum_{w=x,y,z}\Big\langle  \frac{d}{dt} \HH_w^{t},  \DD_k \frac{d}{dt} \HH_w^{t}\Big\rangle_N+
\frac{\eps}{2}\sum_{w=x,y,z}\Big\langle \frac{d}{dt} \EE_w^{t},  \DD_k  \frac{d}{dt} \EE_w^{t}\Big\rangle_N\ \ \ \textmd{for}\ \ \ k=1,2,3,
    \end{aligned}
\end{equation*}
with the inner product  $\langle\cdot,\cdot\rangle_N$  and the notations
$\DD_1= I_{N_z}  \otimes I_{N_y} \otimes  D_x,\
\DD_2= I_{N_z}  \otimes  D_y  \otimes  I_{N_x},\
\DD_3=D_z  \otimes I_{N_y} \otimes I_{N_x}.
$
\end{theo}
\begin{proof}
Based on the formulation of fully discrete { scheme \ref{dIUA-PE}}, it is known that    $ \EE^{t_{end}}$ and $ \HH^{t_{end}}$ satisfy
\begin{equation}\label{hamFD}
\frac{d}{d t}  \begin{pmatrix}
                               \sqrt{\mu}\HH^{t_{end}}  \\
                             \sqrt{\eps}\EE^{t_{end}}
                                 \end{pmatrix}
   = \mathcal{D} \begin{pmatrix}
                               \sqrt{\mu} \HH^{t_{end}}   \\
                             \sqrt{\eps}\EE^{t_{end}}
                                 \end{pmatrix}.
\end{equation}
Therefore, we have that
\begin{equation*}
  \begin{aligned}
\frac{d}{d t} \mathcal{E}_1^{t_{end}}=& \mu (\HH^{t_{end}}) ^{\intercal} \dot{\HH}^{t_{end}} + \eps   (\EE^{t_{end}})^{\intercal} \dot{\EE}^{t_{end}}
=  (\HH^{t_{end}}) ^{\intercal} (-\DD \EE^{t_{end}}) + (\EE^{t_{end}})^{\intercal} (\DD \HH^{t_{end}}).
    \end{aligned} \end{equation*}
Using the property that $\DD^{\intercal}=\DD$, one gets
 \begin{equation*}
  \begin{aligned}
\frac{d}{d t} \mathcal{E}_1^{t_{end}}
=  (\HH^{t_{end}}) ^{\intercal} (-\DD \EE^{t_{end}}) +\Big( (\EE^{t_{end}})^{\intercal} (\DD \HH^{t_{end}})\Big)^{\intercal}= (\HH^{t_{end}}) ^{\intercal} (-\DD \EE^{t_{end}}) + (\HH^{t_{end}})^{\intercal} (\DD \EE^{t_{end}})=0.
    \end{aligned} \end{equation*}

    From \eqref{hamFD}, it follows that
$
\frac{d^2}{d t^2}  \begin{pmatrix}
                               \sqrt{\mu}\HH^{t_{end}}  \\
                             \sqrt{\eps}\EE^{t_{end}}
                                 \end{pmatrix}
   = \mathcal{D} \frac{d}{d t}  \begin{pmatrix}
                               \sqrt{\mu} \HH^{t_{end}}   \\
                             \sqrt{\eps}\EE^{t_{end}}
                                 \end{pmatrix}
$ {and hence}
\begin{equation*}
  \begin{aligned}
\frac{d}{d t} \mathcal{E}_2^{t_{end}}=& \mu (\dot{\HH}^{t_{end}}) ^{\intercal} \ddot{\HH}^{t_{end}} + \eps   (\dot{\EE}^{t_{end}})^{\intercal} \ddot{\EE}^{t_{end}}
=(\dot{\HH}^{t_{end}}) ^{\intercal} (-\DD \dot{\EE}^{t_{end}}) +\Big( (\dot{\EE}^{t_{end}})^{\intercal} (\DD \dot{\HH}^{t_{end}})\Big)^{\intercal}\\
=& (\dot{\HH}^{t_{end}}) ^{\intercal} (-\DD \dot{\EE}^{t_{end}}) + (\dot{\HH}^{t_{end}})^{\intercal} (\DD \dot{\EE}^{t_{end}})=0.
    \end{aligned} \end{equation*}

  For the block diagonal matrices $\textbf{B}_k=\left(
                                                     \begin{array}{ccc}
                                                       \DD_k &   &   \\
                                                         & \DD_k &   \\
                                                         &   & \DD_k \\
                                                     \end{array}
                                                   \right),$
                                                   it is easy to see that
      \begin{equation*}
  \begin{aligned}
\textbf{B}_k\DD=&\left(
                                                     \begin{array}{ccc}
                                                       \DD_k &   &   \\
                                                         & \DD_k &   \\
                                                         &   & \DD_k \\
                                                     \end{array}
                                                   \right)\left(
        \begin{array}{ccc}
          \textbf{0} & -D_z  \otimes I_{N_y} \otimes I_{N_x} & I_{N_z}  \otimes  D_y  \otimes  I_{N_x}\\
         D_z  \otimes I_{N_y} \otimes I_{N_x} & \textbf{0} & -I_{N_z}  \otimes I_{N_y} \otimes  D_x\\
          -I_{N_z}  \otimes  D_y  \otimes  I_{N_x}  &I_{N_z}  \otimes I_{N_y} \otimes  D_x & \textbf{0} \\
        \end{array}
      \right)\\
      =&\left(
        \begin{array}{ccc}
          \textbf{0} & -D_z  \otimes I_{N_y} \otimes I_{N_x}\DD_k & I_{N_z}  \otimes  D_y  \otimes  I_{N_x}\DD_k\\
         D_z  \otimes I_{N_y} \otimes I_{N_x}\DD_k & \textbf{0} & -I_{N_z}  \otimes I_{N_y} \otimes  D_x\DD_k\\
          -I_{N_z}  \otimes  D_y  \otimes  I_{N_x} \DD_k &I_{N_z}  \otimes I_{N_y} \otimes  D_x \DD_k & \textbf{0} \\
        \end{array}
      \right)=\DD \textbf{B}_k,
    \end{aligned} \end{equation*}
 where we have used     the
commutative law of $\DD_k$ which can be shown for $\DD_2 \DD_1$ as follows:
          \begin{equation*}
  \begin{aligned}
\DD_2 \DD_1 =& ( I_{N_z}  \otimes  D_y  \otimes  I_{N_x})(I_{N_z}  \otimes I_{N_y} \otimes  D_x) \\
=&  I_{N_z}  \otimes \big((D_y  \otimes  I_{N_x})(  I_{N_y} \otimes  D_x)\big) =I_{N_z}  \otimes D_y   \otimes  D_x= \DD_1  \DD_2.
\end{aligned} \end{equation*}
      Then, left-multiplying \eqref{hamFD} with block diagonal matrix $\textmd{diag}(\textbf{B}_k, \textbf{B}_k)$, we have
      \begin{equation}\label{hamFD-1}
\frac{d}{d t}  \begin{pmatrix}
                               \sqrt{\mu}\textbf{B}_k \HH^{t_{end}}  \\
                             \sqrt{\eps}\textbf{B}_k \EE^{t_{end}}
                                 \end{pmatrix}
                                    = \frac{1}{\sqrt{\mu\eps}}\left(
       \begin{array}{cc}
         \mathbf{0} & -\textbf{B}_k \DD\\
         \textbf{B}_k \DD & \mathbf{0} \\
       \end{array}
     \right) \begin{pmatrix}
                               \sqrt{\mu} \HH^{t_{end}}   \\
                             \sqrt{\eps}  \EE^{t_{end}}
                                 \end{pmatrix}
   = \frac{1}{\sqrt{\mu\eps}}\left(
       \begin{array}{cc}
         \mathbf{0} & - \DD\\
         \DD & \mathbf{0} \\
       \end{array}
     \right) \begin{pmatrix}
                               \sqrt{\mu}\textbf{B}_k \HH^{t_{end}}   \\
                             \sqrt{\eps}\textbf{B}_k \EE^{t_{end}}
                                 \end{pmatrix}.
\end{equation}
Based on this scheme and the same arguments as  $  \mathcal{E}_1^{t_{end}}$, it is obtained that  $\frac{d}{d t} \mathcal{E}_3^{t_{end}}=0.$

The  statement  of $\mathcal{E}_4^{t_{end}}$    can be proved by combining the proofs of $\mathcal{E}_2^{t_{end}}$ and $\mathcal{E}_3^{t_{end}}$.

   For the energy $ \mathcal{E}_5^{t_{end}}$,  it is deduced that
    \begin{equation*}
  \begin{aligned}
\frac{d}{d t} \mathcal{E}_5^{t_{end}}=& \mu    \sum_{w=x,y,z} (\HH_w^{t_{end}}) ^{\intercal}\DD_k  \dot{\HH}_w^{t_{end}} + \eps    \sum_{w=x,y,z} (\EE_w^{t_{end}}) ^{\intercal}\DD_k  \dot{\EE}_w^{t_{end}}\\
=&\mu     (\HH^{t_{end}}) ^{\intercal}\textbf{B}_k  \dot{\HH}^{t_{end}} + \eps      (\EE^{t_{end}}) ^{\intercal}\textbf{B}_k \dot{\EE}^{t_{end}}
=    - (\HH^{t_{end}}) ^{\intercal}\DD\textbf{B}_k \EE^{t_{end}} +      (\EE^{t_{end}}) ^{\intercal}\DD\textbf{B}_k \HH^{t_{end}}=0.
    \end{aligned} \end{equation*}

The last result of $\mathcal{E}_6^{t_{end}}$  can be proved in a
similar way {to that stated above.} \hfill
\end{proof}

\begin{rem}
It is noted that the first derivatives of $\HH^{t_{end}}, \EE^{t_{end}}$ are needed in the results and they are obtained by  \begin{equation*}
 \dot{\EE}^{t_{end}}_w=\big(\mathcal{F}^{-1}_{N_z}  \otimes \mathcal{F}^{-1}_{N_y} \otimes \mathcal{F}^{-1}_{N_x}\big)\dot{\widehat{\EE}}^{t_{end}}_w, \quad \dot{\HH}^{t_{end}}_w=\big(\mathcal{F}^{-1}_{N_z}  \otimes \mathcal{F}^{-1}_{N_y} \otimes \mathcal{F}^{-1}_{N_x}\big)\dot{\widehat{\HH}}^{t_{end}}_w,\ \ \textmd{for}\ \ w=x,y,z,
\end{equation*}
with {
 \begin{equation*}
\begin{array}[c]{ll}%
&\dot{\widehat{\HH}}_x^{t_{end}}=
-\frac{\ii}{\mu}\big(-\Omega_3\widehat{\EE}_y^{t_{end}}+\Omega_2\widehat{\EE}_z^{t_{end}}\big),\ \
 \dot{\widehat{\HH}}_y^{t_{end}}=-\frac{\ii}{\mu}\big( \Omega_3\widehat{\EE}_x^{t_{end}}-\Omega_1\widehat{\EE}_z^{t_{end}}\big),\\
&\dot{\widehat{\HH}}_z^{t_{end}}=-\frac{\ii}{\mu}\big(-\Omega_2\widehat{\EE}_x^{t_{end}}
+\Omega_1\widehat{\EE}_y^{t_{end}}\big),\ \
 \dot{\widehat{\EE}}_x^{t_{end}}=\frac{\ii}{\eps}\big(-\Omega_3\widehat{\HH}_y^{t_{end}}+\Omega_2\widehat{\HH}_z^{t_{end}}\big),\\
&\dot{\widehat{\EE}}_y^{t_{end}}=\frac{\ii}{\eps}\big( \Omega_3\widehat{\HH}_x^{t_{end}}-\Omega_1\widehat{\HH}_z^{t_{end}}\big),\qquad \ \
 \dot{\widehat{\EE}}_z^{t_{end}}=\frac{\ii}{\eps}\big(-\Omega_2\widehat{\HH}_x^{t_{end}}+
 \Omega_1\widehat{\HH}_y^{t_{end}}\big).
\end{array}\end{equation*}}
On the other hand, these discrete energy conservation laws imply that the numerical solutions are
bounded in the $L^2$ norm and do not blow up. Therefore, the scheme proposed in the paper is unconditionally stable.
\end{rem}

\begin{theo} \label{ene law thm} 
\textbf{(Helicity conservation laws.)}
For the solutions  given by the scheme \ref{dIUA-PE},  two
discrete Helicity conservation laws
$\mathcal{H}_1^{t_{end}}=\mathcal{H}_1^{0},\
  \mathcal{H}_2^{t_{end}}=\mathcal{H}_2^{0}$ hold,
where
\begin{equation*}
  \begin{aligned}& \mathcal{H}_1^{t}=\frac{1}{2\eps}  \langle\HH^{t},\DD \HH^{t}\rangle_N+\frac{1}{2\mu}  \langle\EE^{t},\DD \EE^{t}\rangle_N,\ \
  \ \mathcal{H}_1^{t}=\frac{1}{2\eps}  \langle \frac{d}{dt}\HH^{t},\DD  \frac{d}{dt}\HH^{t}\rangle_N+\frac{1}{2\mu}  \langle \frac{d}{dt}\EE^{t},\DD  \frac{d}{dt}\EE^{t}\rangle_N.
    \end{aligned}
\end{equation*}
\end{theo}
\begin{proof}
Using \eqref{hamFD}, it is arrived at
  \begin{equation*}
  \begin{aligned}
\frac{d}{d t} \mathcal{H}_1^{t_{end}}
= &\frac{1}{\eps } (\HH^{t_{end}}) ^{\intercal}\DD \dot{\HH}^{t_{end}} +  \frac{1}{\mu}  (\EE^{t_{end}})^{\intercal}\DD \dot{\EE}^{t_{end}}
= \frac{1}{\eps \mu} (\HH^{t_{end}}) ^{\intercal}\DD (-\DD \EE^{t_{end}}) +  \frac{1}{\mu \eps}  (\EE^{t_{end}})^{\intercal}\DD (\DD \HH^{t_{end}})\\
=&- \frac{1}{\eps \mu} (\HH^{t_{end}}) ^{\intercal}\DD^2 \EE^{t_{end}} +  \frac{1}{\mu \eps}  (\EE^{t_{end}})^{\intercal}\DD^2 \HH^{t_{end}}=0.
    \end{aligned} \end{equation*}
The second  statement  can be proved by  the similar arguments
{to} $\mathcal{H}_1^{t_{end}}$. \hfill
\end{proof}

\begin{theo} \label{mom law thm} 
\textbf{(Momentum conservation laws.)}
The solutions of the proposed scheme \ref{dIUA-PE} possess the
discrete momentum conservation laws $\mathcal{M}_1^{t_{end}}=\mathcal{M}_1^{0}, \ \mathcal{M}_2^{t_{end}}=\mathcal{M}_2^{0},$
where for $k=1,2,3$
\begin{equation*}
  \begin{aligned}&  \mathcal{M}_1^{t}=  \langle\HH^{t},\textbf{B}_k\EE^{t}\rangle_N,\quad  \mathcal{M}_2^{t}=  \langle\EE^{t},\textbf{B}_k\HH^{t}\rangle_N \   \ \textmd{with}\ \
  \textbf{B}_k=\textmd{diag}(\DD_k, \DD_k,\DD_k).
    \end{aligned}
\end{equation*}
\end{theo}
\begin{proof}We first compute
  \begin{equation*}
  \begin{aligned}
\frac{d}{d t} \mathcal{M}_1^{t_{end}}
= &  (\HH^{t_{end}}) ^{\intercal}\textbf{B}_k\dot{\EE}^{t_{end}} +(\dot{\HH}^{t_{end}}) ^{\intercal}  \textbf{B}_k\EE^{t_{end}}
=  \frac{1}{\eps} (\HH^{t_{end}}) ^{\intercal}\textbf{B}_k\DD \HH^{t_{end}} {-}\frac{1}{\mu}( \DD\EE^{t_{end}}) ^{\intercal}  \textbf{B}_k\EE^{t_{end}}.
    \end{aligned} \end{equation*}
With the properties  { $\DD^{\intercal}  =\DD,\ \textbf{B}_k^{\intercal}=-\textbf{B}_k$ and
    $\textbf{B}_k\DD=\DD \textbf{B}_k$, one has $(\textbf{B}_k\DD )^{\intercal}=-\textbf{B}_k\DD$}. Thus it is clear that
    $\frac{d}{d t} \mathcal{M}_1^{t_{end}}=0.$ The other statement can be shown in the same way.\hfill
\end{proof}

\begin{theo} \label{sym law thm} 
\textbf{(Symplecticity conservation law.)}
The solutions of the  scheme \ref{dIUA-PE} have the
discrete symplecticity conservation law $ d \EE^{t_{end}} \wedge  d \HH^{t_{end}}=d \EE^{0} \wedge  d \HH^{0},$
where
\begin{equation*}
d \EE^{t} \wedge  d \HH^{t}=d \EE_x^{t} \wedge  d \HH_x^{t}+d \EE_y^{t} \wedge  d \HH_y^{t}+d \EE_z^{t} \wedge  d { \HH_z^{t}.}
\end{equation*}
\end{theo}
\begin{proof}
Considering \eqref{hamFD}, we deduce that
\begin{equation}
   \begin{pmatrix}
                               \sqrt{\mu}\HH^{t_{end}}  \\
                             \sqrt{\eps}\EE^{t_{end}}
                                 \end{pmatrix}
   =e^{t_{end}\mathcal{D}}\begin{pmatrix}
                               \sqrt{\mu} \HH^{0}   \\
                             \sqrt{\eps}\EE^{0}
                                 \end{pmatrix}=\left(
                \begin{array}{cc}
                  \cos\big(\frac{t_{end}}{\sqrt{\mu\eps}}\DD\big) & -\sin\big(\frac{t_{end}}{\sqrt{\mu\eps}}\DD\big) \\
                  \sin\big(\frac{t_{end}}{\sqrt{\mu\eps}}\DD\big) & \cos\big(\frac{t_{end}}{\sqrt{\mu\eps}}\DD\big) \\
                \end{array}
              \right)\begin{pmatrix}
                               \sqrt{\mu} \HH^{0}   \\
                             \sqrt{\eps}\EE^{0}
                                 \end{pmatrix}.
\end{equation}
Therefore, one gets
\begin{equation*}
  \begin{aligned}& d \EE^{t_{end}} \wedge  d \HH^{t_{end}}\\=&\frac{1}{ \sqrt{\eps}}d\Big( \sin\big(\frac{t_{end}}{\sqrt{\mu\eps}}\DD\big) \sqrt{\mu} \HH^{0} +\cos\big(\frac{t_{end}}{\sqrt{\mu\eps}}\DD\big) \sqrt{\eps}\EE^{0}\Big)
  \wedge
  \frac{1}{ \sqrt{\mu}}d\Big( \cos\big(\frac{t_{end}}{\sqrt{\mu\eps}}\DD\big) \sqrt{\mu} \HH^{0} -\sin\big(\frac{t_{end}}{\sqrt{\mu\eps}}\DD\big) \sqrt{\eps}\EE^{0}\Big)\\
  =& \cos^2\big(\frac{t_{end}}{\sqrt{\mu\eps}}\DD\big)d\EE^{0}   \wedge  d\HH^{0}
  - \sin^2\big(\frac{t_{end}}{\sqrt{\mu\eps}}\DD\big)d\HH^{0}   \wedge  d\EE^{0}=d\EE^{0}   \wedge  d\HH^{0},\end{aligned}
\end{equation*}
where we have used the fact that $d\HH^{0}   \wedge  d\EE^{0}=-d\EE^{0}   \wedge  d\HH^{0}.$\hfill
\end{proof}

\begin{theo} \label{divlaw thm} 
\textbf{(Divergence-free field conservation law.)} The following
discrete divergence-free field conservation laws {hold
true} for the solutions $\EE^{t_{end}},
   \HH^{t_{end}}$ produced by the scheme \ref{dIUA-PE}
\begin{equation*}
  \begin{aligned}&  \widetilde{\nabla}\cdot (\eps \EE^{t_{end}})= \widetilde{\nabla}\cdot (\eps \EE^{0}),\ \ \
   \widetilde{\nabla}\cdot (\mu \HH^{t_{end}})= \widetilde{\nabla}\cdot (\mu \HH^{0}),\end{aligned}
\end{equation*}
where $\widetilde{\nabla}\cdot (\eps \EE^{t})= \DD_1 (\eps \EE_x^{t})+ \DD_2 (\eps \EE_y^{t})+ \DD_3 (\eps \EE_z^{t}),\
  \widetilde{\nabla}\cdot (\mu \HH^{t})= \DD_1 (\mu \HH_x^{t})+ \DD_2 (\mu \HH_y^{t})+ \DD_3 (\mu \HH_z^{t}).$
\end{theo}
\begin{proof}{ Concerning the results of $\DD_1,\DD_2,\DD_3$ and $\EE^{t_{end}}$, it} can be verified that
  \begin{equation*}
  \begin{aligned}
\widetilde{\nabla}\cdot   \EE^{t_{end}} =& \DD_1   \EE_x^{t_{end}} + \DD_2   \EE_y^{t_{end}} + \DD_3   \EE_z^{t_{end}} \\
=& \mathcal{F}_{N_S}^{-1}\big(\Lambda_1\textbf{c}_{11}+\Lambda_2\textbf{c}_{12}+\Lambda_3\textbf{c}_{13}\big)  \mathcal{F}_{N_S} \EE_x^{0}+\mathcal{F}_{N_S}^{-1}\big(\Lambda_1\textbf{c}_{12}+\Lambda_2\textbf{c}_{22}+\Lambda_3\textbf{c}_{23}\big)  \mathcal{F}_{N_S} \EE_y^{0}\\&+\mathcal{F}_{N_S}^{-1}\big(\Lambda_1\textbf{c}_{13}+\Lambda_2\textbf{c}_{23}+\Lambda_3\textbf{c}_{33}\big)  \mathcal{F}_{N_S} { \EE_z^{0}}+\frac{\sqrt{\mu}}{\sqrt{\eps}}
\mathcal{F}_{N_S}^{-1}\big(-\Lambda_1\textbf{s}_{12}+\Lambda_3\textbf{s}_{23}\big)  \mathcal{F}_{N_S} \HH_x^{0}\\
&+\frac{\sqrt{\mu}}{\sqrt{\eps}}
\mathcal{F}_{N_S}^{-1}\big(\Lambda_2\textbf{s}_{12}-\Lambda_3\textbf{s}_{13}\big)  \mathcal{F}_{N_S} \HH_y^{0}+
+\frac{\sqrt{\mu}}{\sqrt{\eps}}
\mathcal{F}_{N_S}^{-1}\big(\Lambda_1\textbf{s}_{13}-\Lambda_2\textbf{s}_{23}\big)  \mathcal{F}_{N_S} \HH_z^{0},
    \end{aligned} \end{equation*}
    where $\mathcal{F}_{N_S}=\mathcal{F}_{N_z}  \otimes \mathcal{F}_{N_y} \otimes \mathcal{F}_{N_x}$ and we omit $(t_{end}\Lambda/\sqrt{\mu\eps})$ for brevity.
    According to the results given in the formulation of the method, it can be checked that
      \begin{equation*}
  \begin{aligned}
& \Lambda_1\textbf{c}_{11}+\Lambda_2\textbf{c}_{12}+\Lambda_3\textbf{c}_{13}=\mathbf{I},\ \ \Lambda_1\textbf{c}_{12}+\Lambda_2\textbf{c}_{22}+\Lambda_3\textbf{c}_{23} =\mathbf{I},\ \ \Lambda_1\textbf{c}_{13}+\Lambda_2\textbf{c}_{23}+\Lambda_3\textbf{c}_{33}=\mathbf{I},\\
&-\Lambda_1\textbf{s}_{12}+\Lambda_3\textbf{s}_{23}=\mathbf{0},\qquad \quad
 \Lambda_2\textbf{s}_{12}-\Lambda_3\textbf{s}_{13}=\mathbf{0},\qquad \qquad \
\Lambda_1\textbf{s}_{13}-\Lambda_2\textbf{s}_{23}=\mathbf{0},
    \end{aligned} \end{equation*}
   which lead to the first result of this theorem. The second one can be proved in a similar way.\hfill
\end{proof}

\begin{rem}
These conservation laws stated above are established in a discrete
form, i.e., those invariants are defined by  $\HH^{t}$ and
$\EE^{t}$. For the error between the discrete conservations and the
exact conservations, it can be obtained by considering   the
convergence shown in Theorem  \ref{con thm}, which leads to the
estimate $\mathcal{O}(N^{-r})$. {For instance, we} take
the discrete divergence-free field conservation laws and it can be
shown that $\abs{\widetilde{\nabla}\cdot (\eps \EE^{0})-\nabla\cdot
(\eps \EE^{0})}\leq CN^{-r}.$
   \end{rem}

\section{Numerical experiments}\label{sec:5}
 In this section, we present numerical experiments
to show the performance of  our scheme. These two tests are conducted in a sequential program in MATLAB on a laptop
   ThinkPad X1 Nano (CPU: 11th Gen Intel(R) Core(TM) i7-1160G7 @ 1.20GHz   2.11 GHz, Memory: 16 GB, Os: Microsoft Windows 11 with 64bit).

   \subsection{Standing wave solutions}
The first test is devoted to the standing wave solutions of Maxwell's equations
\eqref{maxsys} (\cite{n5,n14})
   \begin{equation}\label{standing}
  \begin{aligned}&E_x=\frac{k_y-k_z}{\eps\sqrt{\mu} \omega}\cos(\omega\pi t)\cos(k_x\pi x)
  \sin(k_y\pi y)  \sin(k_z\pi z),\ \
    H_x=\sin(\omega\pi t)\sin(k_x\pi x)
  \cos(k_y\pi y)  \cos(k_z\pi z),\\
    &E_y=\frac{k_z-k_x}{\eps\sqrt{\mu} \omega}\cos(\omega\pi t)\sin(k_x\pi x)
  \cos(k_y\pi y)  \sin(k_z\pi z),\ \
   H_y=\sin(\omega\pi t)\cos(k_x\pi x)
  \sin(k_y\pi y)  \cos(k_z\pi z),\\
    &E_z=\frac{k_x-k_y}{\eps\sqrt{\mu} \omega}\cos(\omega\pi t)\sin(k_x\pi x)
  \sin(k_y\pi y)  \cos(k_z\pi z),\  \
     H_z=\sin(\omega\pi t)\cos(k_x\pi x)
  \cos(k_y\pi y)  \sin(k_z\pi z),
  \end{aligned}
\end{equation}
where $\eps=\mu=1,\ \omega=\sqrt{\frac{k_x^2+k_y^2+k_z^2}{\eps\mu}},\ k_x=1,\ k_y=2,\ k_z=-3$ and $ \Omega=[0,2]^3$.

 \textbf{Energy conservation behaviour.}
 In this part, we test the performance of our scheme in the structure preserving laws, which begins with the energy invariants.
Define the relative errors in discrete energy invariants  $\textmd{Re}(\mathcal{E}_k)=\frac{\abs{\mathcal{E}_k^{t_{end}}-\mathcal{E}_k^{0}}}{\abs{\mathcal{E}_k^{0}}}$
and display the results of our scheme for $k=1,2,3,4$  in Table \ref{error e}. We note here that the scheme has a similar behaviour for the other two  energy invariants
$\mathcal{E}_5, \mathcal{E}_6$ and the corresponding results are skipped for brevity.
Then we present in   Table \ref{error o} the relative changes in discrete helicity and momentum invariants
$\textmd{Re}(\mathcal{H}_k)=\frac{\abs{\mathcal{H}_k^{t_{end}}-\mathcal{H}_k^{0}}}{\abs{\mathcal{H}_k^{0}}}$
and $\textmd{Re}(\mathcal{M}_k)=\frac{\abs{\mathcal{M}_k^{t_{end}}-\mathcal{M}_k^{0}}}{\abs{\mathcal{M}_k^{0}}}$
 {for} $k=1,2.$
 Finally, the   relative errors in divergence-free field discrete helicity
  $\textmd{Re}(\mathcal{D}_1)= \abs{\widetilde{\nabla}\cdot (\eps \HH^{t_{end}})}$ and  $\textmd{Re}(\mathcal{D}_2)= \abs{\widetilde{\nabla}\cdot (\eps \EE^{t_{end}})}$
 are displayed in Table \ref{error D}. It can be observed from the results that our scheme  preserves the  invariants exactly since the relative errors are within the roundoff error
 of the machine, which supports the theoretical analysis proposed in {this paper.} To show the long time conservation, we plot the errors in discrete energy invariants
 on $[0,10000]$ and the results are shown in Figure \ref{fig21}. {It can be seen} that our scheme has a persistent conservation over long times.

 \textbf{Accuracy analysis.}   For this problem, the regularities are infinite so that as shown in Theorem  \ref{con thm} the solutions of  the  schemes converge with infinite-order accuracy both in space and in time.
  In our simulations, we use two  norms which are defined by
  $L_{2}= \Big(  \norm{\EE^{t_{end}}-\EE(t_{end})}_N^2+\norm{\HH^{t_{end}}-\HH(t_{end})}_N^2\Big)^{\frac{1}{2}}$
  and  $L_{\infty}=\max\{ \max\abs{\EE^{t_{end}}-\EE(t_{end})}, \max\abs{\HH^{t_{end}}-\HH(t_{end})}\}$
to scale the maximal and average errors in solution, respectively.
The numerical errors as well as the CPU time used in the scheme are
listed in Table \ref{error s}.  From the results, it can be observed
that our scheme provides {numerical results with} a
machine accuracy and the {computational cost} is very
low, which demonstrates the effectiveness and efficiency of the new
scheme.


\begin{table}[tbp]
\caption{The relative errors in discrete energy invariants for Example 1.}\label{error e2}
  \vskip0.1in \tabcolsep0.1in
\begin{tabular}{ l lllllll }
    \toprule
$\raisebox{-0.0ex}[0pt]{Spatial grid points}$   & $\raisebox{-0.0ex}[0pt]{Time}$   & $\textmd{Re}(\mathcal{E}_1)$  &  $\textmd{Re}(\mathcal{E}_2)$    & $\textmd{Re}(\mathcal{E}_3)$
&$\textmd{Re}(\mathcal{E}_4)$     \\
\hline
$N_x=N_y=N_z=8$
              &$t_{end}=$1 &   2.9605e-16  & 2.7425e-16      & 2.1331e-16  & 0.0000e-16   \\
              &$t_{end}=$5 &   1.4802e-16  & 1.3712e-16      & 2.1331e-16  & 3.9521e-16    \\
              &$t_{end}=$10 &   5.9211e-16 & 4.1138e-16      & 4.2662e-16  & 5.9281e-16   \\
              &$t_{end}=$15 &   1.4802e-16  & 1.3712e-16      & 0.0000e-16  & 1.9760e-16      \\
              &$t_{end}=$20 &   2.9605e-16  & 4.1138e-16      & 4.2662e-16  & 1.9760e-16      \\
              \hline
$N_x=N_y=N_z=16$
              &$t_{end}=$1 &   2.9605e-16  & 2.7425e-16      & 2.1331e-16  &  0.0000e-16   \\
              &$t_{end}=$5 &   0.0000e-16  &1.3712e-16      & 4.2662e-16  &  1.9760e-16    \\
              &$t_{end}=$10 &   2.9605e-16 & 4.1138e-16      & 4.2662e-16  & 3.9521e-16   \\
              &$t_{end}=$15 &  0.0000e-16  & 0.0000e-16      & 0.0000e-16  & 0.0000e-16      \\
              &$t_{end}=$20 &   1.4802e-16  & 1.3712e-16      & 4.2662e-16  & 1.9760e-16      \\
    \bottomrule
\end{tabular}
\end{table}

\begin{table}[t!]
\caption{The relative errors in discrete  helicity and momentum invariants for Example 1.}\label{error o2}
  \vskip0.1in \tabcolsep0.1in
\begin{tabular}{ l lllllll }
    \toprule
$\raisebox{-0.0ex}[0pt]{Spatial grid points}$   & $\raisebox{-0.0ex}[0pt]{Time}$   & $\textmd{Re}(\mathcal{H}_1)$  &  $\textmd{Re}(\mathcal{H}_2)$    & $\textmd{Re}(\mathcal{M}_1)$
&$\textmd{Re}(\mathcal{M}_2)$     \\
\hline
$N_x=N_y=N_z=8$
              &$t_{end}=$1 &   34575e-14  & 53756e-14      & 6.6189e-12  & 5.7559e-12   \\
              &$t_{end}=$5 &   10759e-14  & 44125e-14      & 8.2414e-12  & 7.5206e-12    \\
              &$t_{end}=$10 &   53193e-15 & 31357e-15      & 2.3899e-12  & 6.0112e-12   \\
              &$t_{end}=$15 &   14507e-14  & 29453e-14      & 6.4139e-12  & 6.6866e-12     \\
              &$t_{end}=$20 &   42917e-14  & 63387e-14      & 3.7772e-12  & 2.1877e-11      \\
              \hline
$N_x=N_y=N_z=16$
              &$t_{end}=$1 &   2.0793e-14  & 1.6924e-13      & 1.2774e-10  &  5.0600e-10   \\
              &$t_{end}=$5 &   1.3323e-14  & 1.2081e-13      & 1.2704e-10  &  5.6392e-10    \\
              &$t_{end}=$10 &  6.1976e-14 & 1.3128e-13      & 2.6356e-11  & 5.8609e-10   \\
              &$t_{end}=$15 &  6.4600e-14  & 1.0865e-13      & 1.7854e-10  & 5.4005e-10      \\
              &$t_{end}=$20 &   4.8450e-15  & 2.1487e-13      & 1.2504e-10  & 3.4773e-10      \\
    \bottomrule
\end{tabular}
\end{table}

\begin{table}[tbp]
\caption{The  errors in  discrete divergence-free field conservation for Example 1.}\label{error D2}
\centering
\begin{threeparttable}
\begin{tabular}{lccccc}
    \toprule
      \multicolumn{1}{c}{  }&   \multicolumn{2}{c}{$N_x=N_y=N_z=8$ }&\multicolumn{2}{c}{$N_x=N_y=N_z=16$ }\cr
    \cmidrule(lr){2-3} \cmidrule(lr){4-5}
    Time&$\textmd{Re}(\mathcal{D}_1)$&$\textmd{Re}(\mathcal{D}_2)$&$\textmd{Re}(\mathcal{D}_1)$&$\textmd{Re}(\mathcal{D}_2)$\cr
    \midrule
     $t_{end}=$1 &  2.5121e-15 &  4.3962e-14      &  7.1054e-15  & 1.4210e-14\cr
     $t_{end}=$5 &  2.5121e-15  & 4.3962e-14     &  1.4210e-14  & 1.7763e-14\cr
     $t_{end}=$10 &  5.0242e-15  &4.3962e-14      &  1.4210e-14  & 1.3322e-14\cr
     $t_{end}=$15 &  2.5121e-15  & 2.5121e-14      &  1.0658e-14  & 1.0658e-14\cr
     $t_{end}=$20 &   2.5121e-15  &3.7682e-14      &  1.4210e-14  & 1.0658e-14\cr
    \bottomrule
    \end{tabular}
    \end{threeparttable}
\end{table}

    \begin{figure}[tbp]
$$\begin{array}{cc}
\psfig{figure=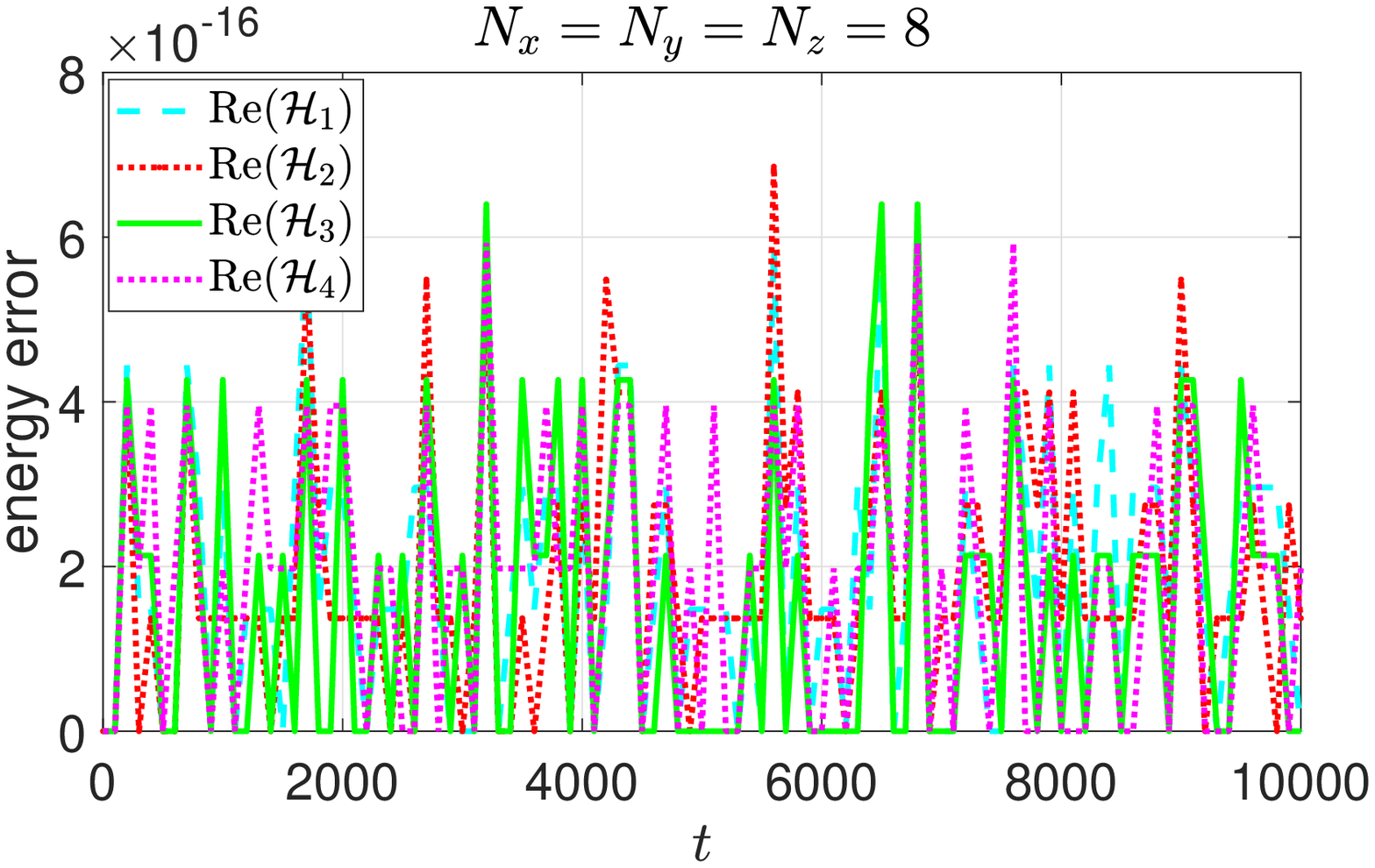,height=3.9cm,width=7.9cm}
\psfig{figure=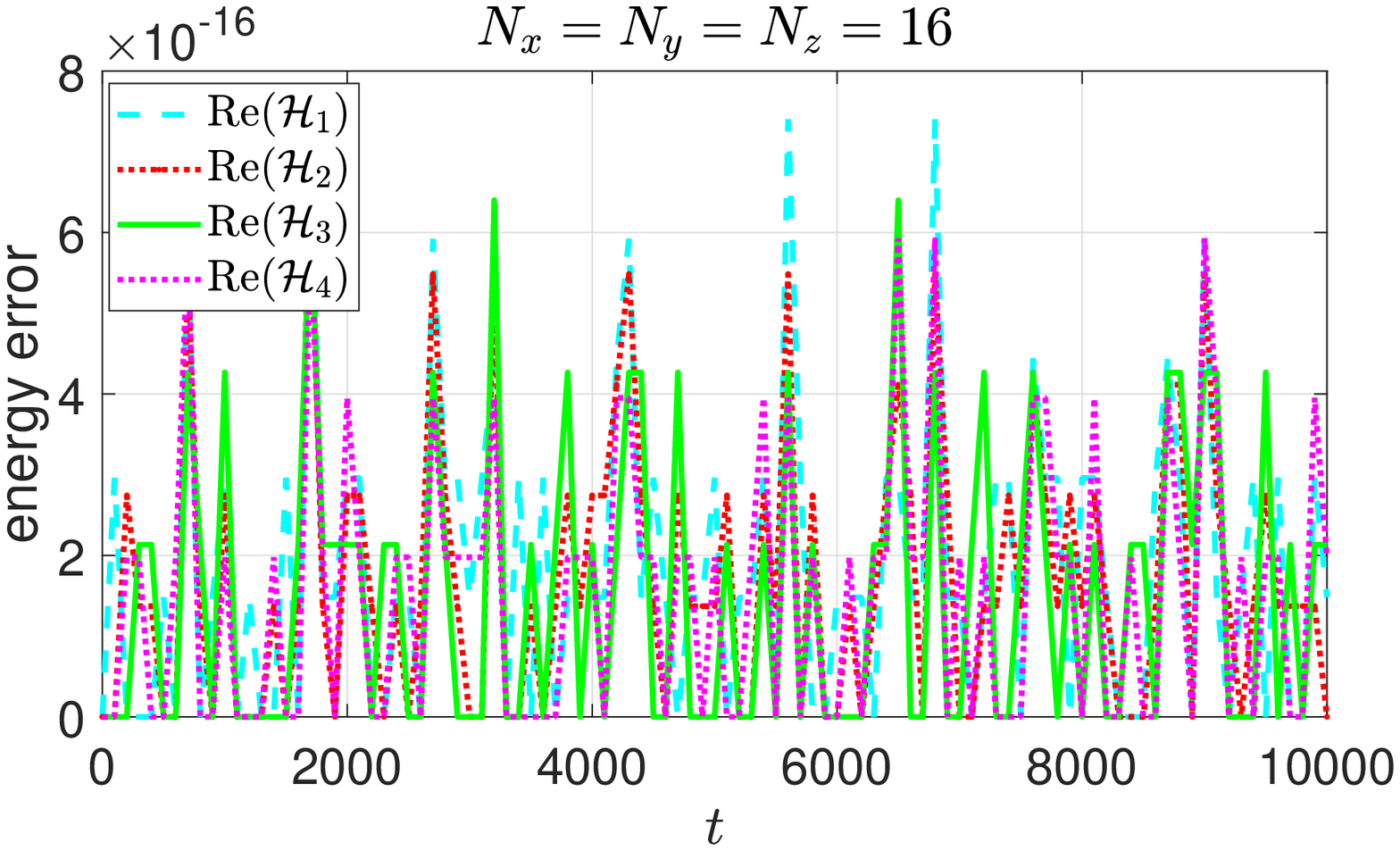,height=3.9cm,width=7.9cm}
\end{array}$$
\caption{Numerical energy conservations of the schemes  over long times  for Example 1.}\label{fig21}
\end{figure}

\begin{table}[tbp]
\caption{The  errors in the solution for Example 1.}\label{error s2}
\centering
\begin{threeparttable}
\begin{tabular}{lccccccc}
    \toprule
      \multicolumn{1}{c}{  }&   \multicolumn{3}{c}{$N_x=N_y=N_z=8$ }&\multicolumn{3}{c}{$N_x=N_y=N_z=16$ }\cr
    \cmidrule(lr){2-4} \cmidrule(lr){5-7}
    Time&$L_{\infty}$&$L_{2}$&CPU (s)&$L_{\infty}$&$L_{2}$&CPU (s)\cr
    \midrule
     $t_{end}=$1 &   2.8588e-13  & 3.6606e-14   &  0.055 & 7.0558e-12  & 3.3915e-13&    0.20\cr
     $t_{end}=$5 &   7.8246e-13  & 1.0241e-13   &  0.023   & 1.1084e-11  & 4.4399e-13&  0.16\cr
     $t_{end}=$10 &   1.4865e-12  & 1.9932e-13   & 0.013   &1.3474e-11  & 6.2014e-13&  0.16\cr
     $t_{end}=$15 &   2.6716e-12  & 3.9738e-13  &  0.015   & 2.1543e-11  & 1.1193e-12& 0.17\cr
     $t_{end}=$20 &   2.7995e-12  & 3.9228e-13  &  0.030    &2.5181e-11  & 1.2169e-12&  0.17\cr
    \bottomrule
    \end{tabular}
    \end{threeparttable}
\end{table}

\subsection{Traveling wave solutions}
The second  numerical example concerns  the
 Maxwell's equations
\eqref{maxsys} which have traveling wave solutions  ($\eps=\mu=1$)
({see} \cite{n5})
   \begin{equation*}
  \begin{aligned}&E_x=\cos(2\pi(x+y+z)-2\sqrt{3}\pi t),\  E_y=-2E_x,\   E_z=E_x,\ H_x=\sqrt{3}E_x, \  H_y=0,\  H_z=-\sqrt{3}E_x,
  \end{aligned}
\end{equation*}
where $t\in[0,t_{end}]$ and $\Omega=[0,1]^3$.

{Tables \ref{error e}-\ref{error D} display the relative
errors in discrete energy invariants,
  discrete  helicity and momentum invariants,  and discrete divergence-free field conservation, respectively.} A long term energy conservation is {indicated} in Figure  \ref{fig11}. The errors in the solution and the corresponding CPU time
are {listed} in Table \ref{error s}.  From the results,
it can be observed that our scheme provides a similar numerical
phenomena {to the first example.}

Finally,  some numerical
comparisons of the existing schemes are made in Table \ref{error sd}. We take some comparisons from \cite{n5}, where
the structure-preserving method \cite{n5}, the ADI-FDTD
method \cite{5}, and the energy conserved splitting FDTD (EC-S-FDTD) method  \cite{n14} are simulated with $N_x=N_y=N_z=32$. For comparison, our scheme (referred as TEIFP)
is implemented with $N_x=N_y=N_z=16$ and the results clearly show that our scheme has an  infinite-order accuracy, which is much better than the existing schemes with finite-order
accuracy.

The numerical results of these two tests highlight the favorable
behavior of {the  scheme presented in this paper.}  It
can be observed that in comparison with other existing
structure-preserving methods, the proposed scheme has very high
accuracy and  exact conservation laws, and requires very low
computing cost.

\begin{table}[tbp]
\caption{The relative errors in discrete energy invariants for Example 2.}\label{error e}
  \vskip0.1in \tabcolsep0.1in
\begin{tabular}{ l lllllll }
    \toprule
$\raisebox{-0.0ex}[0pt]{Spatial grid points}$   & $\raisebox{-0.0ex}[0pt]{Time}$   & $\textmd{Re}(\mathcal{E}_1)$  &  $\textmd{Re}(\mathcal{E}_2)$    & $\textmd{Re}(\mathcal{E}_3)$
&$\textmd{Re}(\mathcal{E}_4)$     \\
\hline
$N_x=N_y=N_z=8$
              &$t_{end}=$1 &   2.9605e-16  & 1.5998e-16      & 1.1998e-16  & 1.2967e-16   \\
              &$t_{end}=$5 &   2.9605e-16  & 0.0000e-16      & 1.1998e-16  & 2.5935e-16    \\
              &$t_{end}=$10 &   0.0000e-16 & 3.1996e-16      & 3.5996e-16  & 1.2967e-16   \\
              &$t_{end}=$15 &   1.4802e-16  & 1.5998e-16      & 1.1998e-16  & 0.0000e-16      \\
              &$t_{end}=$20 &   0.0000e-16  & 3.1996e-16      & 1.1998e-16  & 1.2967e-16      \\
              \hline
$N_x=N_y=N_z=16$
              &$t_{end}=$1 &   0.0000e-16  & 1.5998e-16      & 0.0000e-16  &  0.0000e-16   \\
              &$t_{end}=$5 &   2.9605e-16  & 0.0000e-16      & 2.3997e-16  &  0.0000e-16    \\
              &$t_{end}=$10 &   1.4802e-16 & 1.5998e-16      & 1.1998e-16  & 2.5935e-16   \\
              &$t_{end}=$15 &   2.9605e-16  & 0.0000e-16      & 2.3997e-16  & 1.2967e-16      \\
              &$t_{end}=$20 &   1.4802e-16  & 3.1996e-16      & 0.0000e-16  & 2.5935e-16      \\
    \bottomrule
\end{tabular}
\end{table}

\begin{table}[t!]
\caption{The relative errors in discrete  helicity and momentum invariants for Example 2.}\label{error o}
  \vskip0.1in \tabcolsep0.1in
\begin{tabular}{ l lllllll }
    \toprule
$\raisebox{-0.0ex}[0pt]{Spatial grid points}$   & $\raisebox{-0.0ex}[0pt]{Time}$   & $\textmd{Re}(\mathcal{H}_1)$  &  $\textmd{Re}(\mathcal{H}_2)$    & $\textmd{Re}(\mathcal{M}_1)$
&$\textmd{Re}(\mathcal{M}_2)$     \\
\hline
$N_x=N_y=N_z=8$
              &$t_{end}=$1 &   0.0000e-16  & 0.0000e-16      & 0.0000e-16  & 0.0000e-16   \\
              &$t_{end}=$5 &   0.0000e-16  & 0.0000e-16      & 0.0000e-16  & 0.0000e-16    \\
              &$t_{end}=$10 &   0.0000e-16 & 0.0000e-16      & 0.0000e-16  & 0.0000e-16   \\
              &$t_{end}=$15 &   7.0187e-12  & 8.3126e-10      & 2.9103e-16  & 1.1641e-10      \\
              &$t_{end}=$20 &   0.0000e-16  & 0.0000e-16      & 0.0000e-16  & 0.0000e-16      \\
              \hline
$N_x=N_y=N_z=16$
              &$t_{end}=$1 &   0.0000e-16  & 0.0000e-16      & 0.0000e-16  &  0.0000e-16   \\
              &$t_{end}=$5 &   0.0000e-16  &0.0000e-16      & 0.0000e-16  &  0.0000e-16    \\
              &$t_{end}=$10 &   3.5454e-9 & 1.9930e-7      & 3.7252e-9  & 3.7252e-9   \\
              &$t_{end}=$15 &   2.7610e-9  & 1.0640e-7      & 0.0000e-16  & 7.6798e-9      \\
              &$t_{end}=$20 &   5.1680e-9  & 1.7087e-7     & 3.7252e-9  & 8.3300e-9      \\
    \bottomrule
\end{tabular}
\end{table}

\begin{table}[tbp]
\caption{The  errors in  discrete divergence-free field conservation for Example 2.}\label{error D}
\centering
\begin{threeparttable}
\begin{tabular}{lccccc}
    \toprule
      \multicolumn{1}{c}{  }&   \multicolumn{2}{c}{$N_x=N_y=N_z=8$ }&\multicolumn{2}{c}{$N_x=N_y=N_z=16$ }\cr
    \cmidrule(lr){2-3} \cmidrule(lr){4-5}
    Time&$\textmd{Re}(\mathcal{D}_1)$&$\textmd{Re}(\mathcal{D}_2)$&$\textmd{Re}(\mathcal{D}_1)$&$\textmd{Re}(\mathcal{D}_2)$\cr
    \midrule
     $t_{end}=$1 &   0.0000e-16 &  0.0000e-16      &  0.0000e-16  & 0.0000e-16\cr
     $t_{end}=$5 &  0.0000e-16  &  0.0000e-16      &  0.0000e-16  & 0.0000e-16\cr
     $t_{end}=$10 &  0.0000e-16  &1.0048e-14      &  0.0000e-16  & 1.4210e-14\cr
     $t_{end}=$15 &  0.0000e-16  & 1.0048e-14      &  0.0000e-16  & 2.8421e-14\cr
     $t_{end}=$20 &   0.0000e-16  &0.0000e-16      &  0.0000e-16  & 1.4210e-14\cr
    \bottomrule
    \end{tabular}
    \end{threeparttable}
\end{table}

    \begin{figure}[tbp]
$$\begin{array}{cc}
\psfig{figure=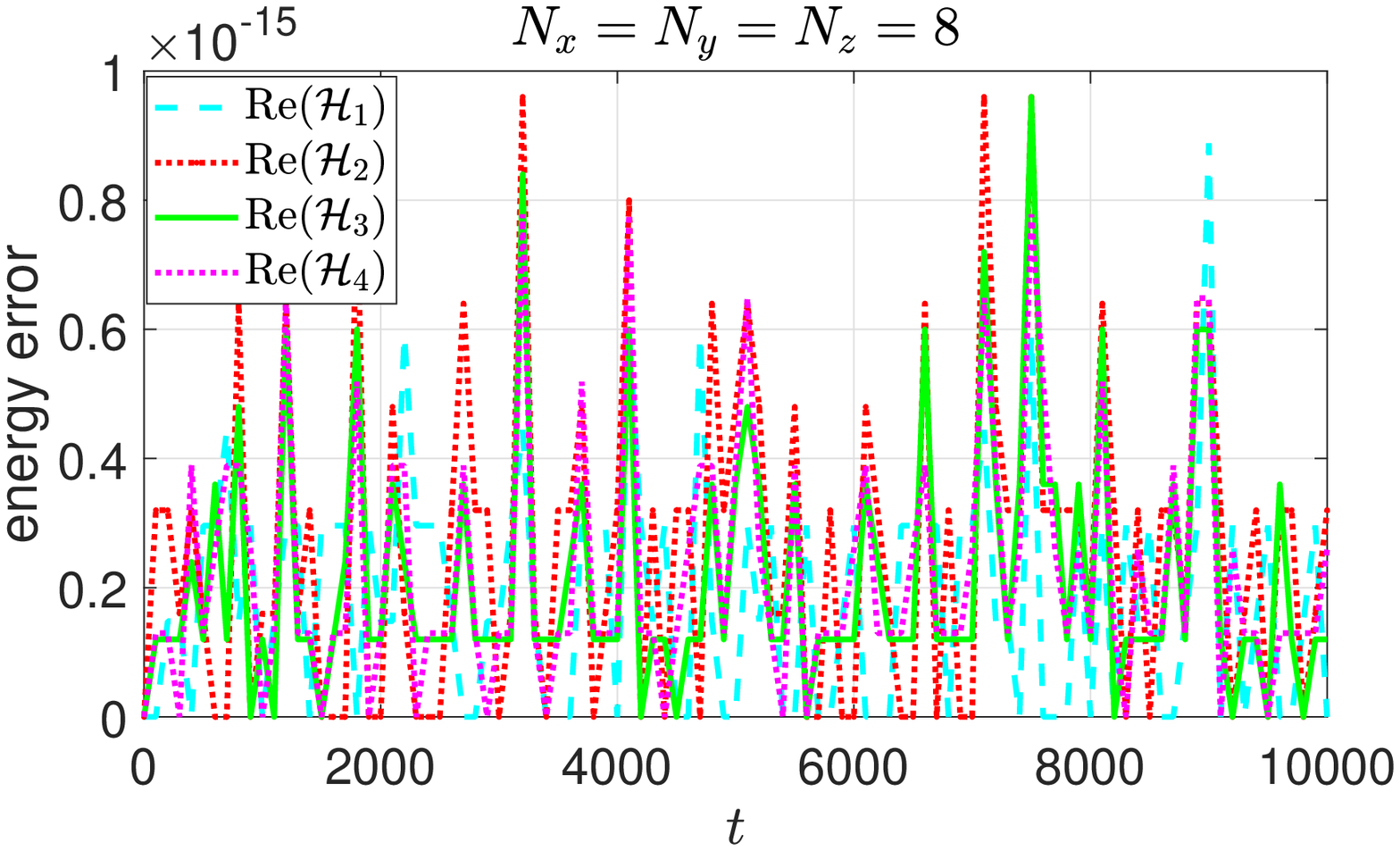,height=3.9cm,width=7.9cm}
\psfig{figure=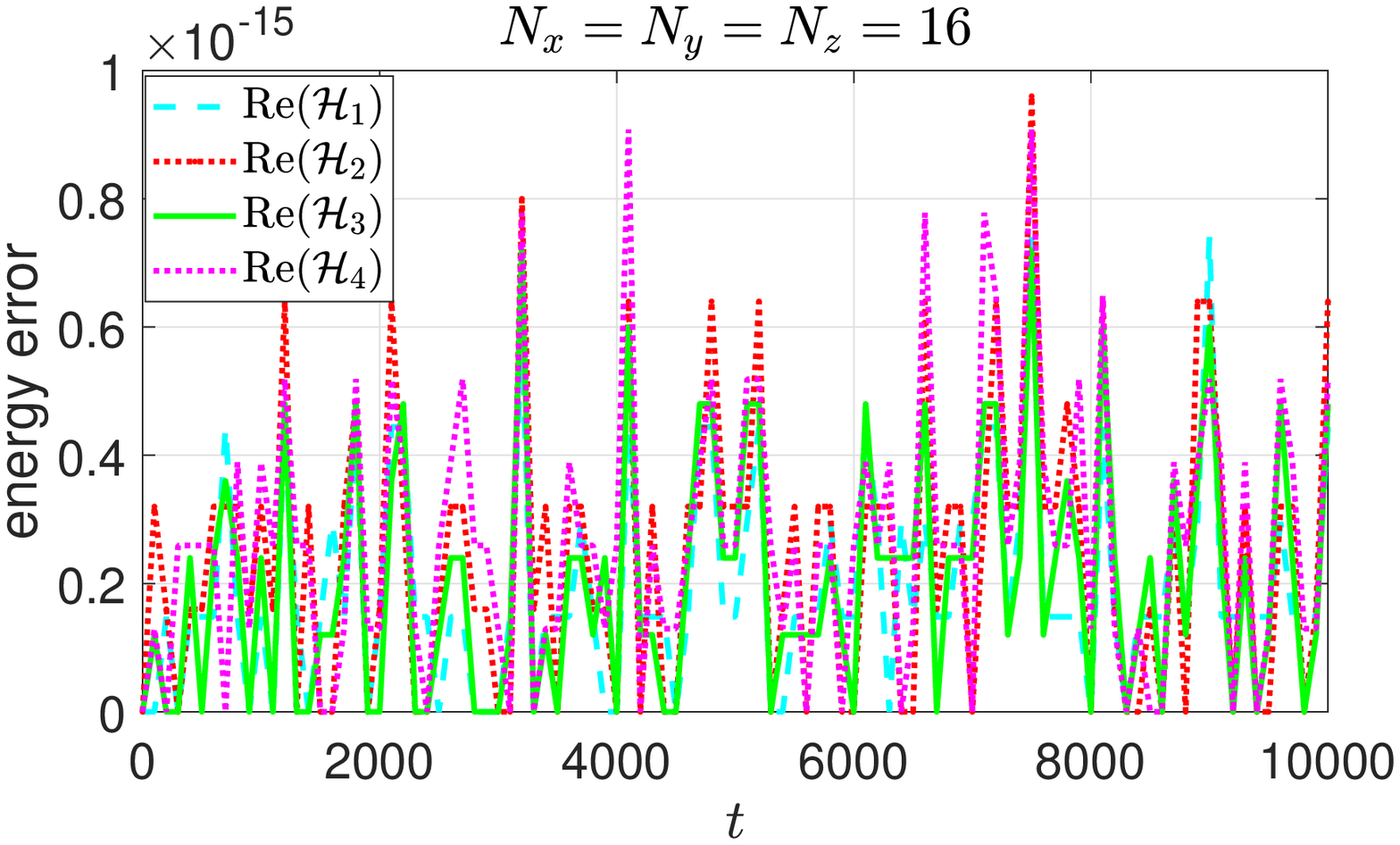,height=3.9cm,width=7.9cm}
\end{array}$$
\caption{Numerical energy conservations of the schemes  over long times for Example 2.}\label{fig11}
\end{figure}

\begin{table}[tbp]
\caption{The  errors in the solution for Example 2.}\label{error s}
\centering

\begin{threeparttable}
\begin{tabular}{lccccccc}
    \toprule
      \multicolumn{1}{c}{  }&   \multicolumn{3}{c}{$N_x=N_y=N_z=8$ }&\multicolumn{3}{c}{$N_x=N_y=N_z=16$ }\cr
    \cmidrule(lr){2-4} \cmidrule(lr){5-7}
    Time&$L_{\infty}$&$L_{2}$&CPU (s)&$L_{\infty}$&$L_{2}$&CPU (s)\cr
    \midrule
     $t_{end}=$1 &   1.6253e-12  & 1.5437e-13 &  0.034    & 3.5895e-10  & 7.4340e-12&  0.16 \cr
     $t_{end}=$5 &   6.6265e-12  & 4.2131e-13 &  0.013     & 3.3974e-10  & 6.7938e-12&   0.15 \cr
     $t_{end}=$10 &   1.4294e-11  & 8.5423e-13& 0.0055       & 3.9946e-10  & 7.6605e-12&  0.14 \cr
     $t_{end}=$15 &   3.0699e-12  & 4.6944e-13 &  0.0033     & 2.8697e-10  & 6.2506e-12&  0.14 \cr
     $t_{end}=$20 &   2.7620e-11  & 1.6616e-12 &  0.0061      & 3.2438e-10  & 7.6824e-12&  0.16 \cr
    \bottomrule
    \end{tabular}
    \end{threeparttable}
\end{table}

\begin{table}[tbp]
\caption{The  errors in the solution for different methods of Example 2.}\label{error sd}
\centering
\begin{threeparttable}
\begin{tabular}{lccccccc}
    \toprule
      \multicolumn{1}{c}{  }&  \multicolumn{2}{c}{$\triangle t=0.05$}
      &\multicolumn{2}{c}{$\triangle t=0.025$ }\cr
    \cmidrule(lr){2-3} \cmidrule(lr){4-5}
    Method&$L_{2}$&Order&$L_{2}$&Order\cr
    \midrule
     SAVF(2) &   3.74e-2  & 2.06 &  9.53e-2 & 1.97   \cr
    ADI-FDTD &   1.76e-1  & 1.79 & 4.53e-2& 1.95 \cr
   EC-S-FDTD &   1.41e-1  & 1.84& 3.62e-2& 1.96  \cr
    TEIFP  &       3.63e-13  & machine accuracy &  3.43e-13     & machine accuracy   \cr
    \bottomrule
    \end{tabular}
    \end{threeparttable}
\end{table}


\begin{thebibliography}{99}


 \bibitem{n5}{\sc J. Cai, J. Hong, Y. Wang, and Y. Gong}, \emph{Two energy-conserved splitting
methods for three-dimensional time-domain Maxwell's equations and the convergence analysis}, SIAM. J. Numer. Anal.,  53 (2015), pp.  1918-1940.

 \bibitem{n6} {\sc J. Cai, J. Hong, Y. Wang, and Y. Gong}, \emph{Numerical analysis of AVF methods for
three-dimensional time-domain Maxwell's equations}, J. Sci. Comput., 66 (2016), pp. 141-176.


 \bibitem{n7}{\sc  W. Cai, Y. Wang, and Y. Song}, \emph{Numerical dispersion analysis of a multi-symplectic scheme for the three dimensional Maxwell's equations}, J. Comput. Phys.,
234 (2013), pp. 330-352.

\bibitem{add3} {\sc C. Canuto and A. Quarteroni}, \emph{Approximation results for orthogonal polynomials
in Sobolev spaces}. Math. Comput., 38 (1982), pp. 67-86,



\bibitem{n13} {\sc W. Chen, X. Li, and D. Liang}, \emph{Energy-conserved splitting FDTD methods for
Maxwell's equations}, Numer. Math., 108 (2008), pp. 445-485.

\bibitem{n14} {\sc W. Chen, X. Li, and D. Liang}, \emph{Energy-conserved splitting finite-difference time-
domain methods for Maxwell's equations in three dimensions}, SIAM. J. Numer.
Anal., 48 (2010), pp. 1530-1554.


 \bibitem{n15} {\sc B. Cockburn, F. Li, and C.W. Shu}, \emph{Locally divergence-free discontinuous Galerkin
methods for the Maxwell equations}, J. Comput. Phys., 194 (2004), pp. 588-610.


\bibitem{Descombes13}{\sc S. Descombes, S. Lanteri, and L. Moya}, \emph{Locally implicit time integration strategies in a
discontinuous Galerkin method for Maxwell's equations}, J. Sci. Comput., 56 (2013), pp. 190-218.

\bibitem{Descombes16} {\sc S. Descombes, S. Lanteri, and L. Moya}, \emph{Locally implicit discontinuous Galerkin time domain
method for electromagnetic wave propagation in dispersive media applied to numerical dosimetry in biological tissues}, SIAM J. Sci. Comput., 38 (2016), pp.  A2611-A2633.

 \bibitem {Diehl10}{\sc  R. Diehl, K. Busch, and J. Niegemann}, \emph{Comparison of low-storage Runge-Kutta schemes
for discontinuous Galerkin time-domain simulations of Maxwell's equations}, J.
Comput. Theo. Nano., 7 (2010), pp.  1572-1580.


  \bibitem{Johannes19}  {\sc J. Eilinghoff, T. Jahnke, and R. Schnaubelt},
\emph{Error analysis of an energy preserving ADI splitting scheme for the Maxwell equations},
SIAM J. Numer. Anal., 57 (2019), pp.   1036-1057.

 \bibitem {Fahs09} {\sc H. Fahs}, \emph{High-order leap-frog based discontinuous Galerkin method for the time-domain
Maxwell equations on non-conforming simplicial meshes}, Numer. Math. Theo. Meth. Appl.,
2 (2009), pp.  275-300.

 \bibitem {6}  {\sc L.  Gao, B. Zhang, and D. Liang}, \emph{The splitting finite-difference time-domain methods for
Maxwell's equations in two dimensions}, J. Comput. Appl. Math., 205 (2007), pp. 207-230.

\bibitem{Grote10}{\sc M. J. Grote and T. Mitkova}, \emph{Explicit local time-stepping methods for Maxwell's equations}, J.
Comput. Appl. Math., 234 (2010), pp.  3283-3302.




\bibitem{LubichW}
{\sc E. Hairer, Ch. Lubich, B. Wang}, \emph{A filtered Boris algorithm for charged-particle dynamics in a strong magnetic field}, Numer. Math., 144 (2020), pp.  787-809.

\bibitem{Lubich}
{\sc E. Hairer, Ch. Lubich, G. Wanner}, \emph{Geometric Numerical Integration: Structure-Preserving Algorithms for Ordinary Differential Equations}, Springer, Berlin, 2006.

\bibitem{Henning16} {\sc  P. Henning, M. Ohlberger, and B. Verf\"{u}rth}, \emph{A new heterogeneous multiscale method for
time-harmonic Maxwell's equations}, SIAM J. Numer. Anal., 54 (2016), pp. 3493-3522,





 \bibitem{n21} {\sc T. Hirono, W. Lui, S. Seki, and Y. Yoshikuni}, \emph{A three-dimensional fourth-order
finite-difference time-domain scheme using a symplectic integrator propagator},
IEEE Trans. Microwave Theory Tech., 49 (2001), pp.   1640-1648.

 \bibitem {Hochbruck151} {\sc M Hochbruck and T. Pa\v{z}ur}, \emph{Implicit Runge-Kutta methods and discontinuous Galerkin dis-
cretizations for linear Maxwell's equations}, SIAM J. Numer. Anal., 53 (2015), pp.  485-507.

 \bibitem {Hochbruck16} {\sc M. Hochbruck and A. Sturm}, \emph{Error analysis of a second-order locally implicit method for
linear Maxwell's equations}, SIAM J. Numer. Anal., 54 (2016), pp.  3167-3191.

 \bibitem {Hochbruck15}{\sc M. Hochbruck, T. Jahnke, and R. Schnaubelt}, \emph{Convergence of an ADI splitting for Maxwell's
equations}, Numeri. Math., 129 (2015), pp.  535-561.


\bibitem{Hochbruck10}{\sc M. Hochbruck and A. Ostermann}, \emph{Exponential integrators}, Acta Numer., 19 (2010), pp.  209-286.

\bibitem{Hochbruck19}{\sc M. Hochbruck, B. Maier, and C. Stohrer}, \emph{Heterogeneous multiscale method for Maxwell's equations}, Multi. Model. Simul., 17 (2019), pp.  1147-1171.

\bibitem{Hochbruck191} {\sc M. Hochbruck and A. Sturm}, \emph{Upwind discontinuous Galerkin space discretization and locally implicit time integration for linear Maxwell's equations}, Math. Comp., 88 (2019), pp.  1121-1153.


\bibitem{n25} {\sc L.  Kong, J.  Hong, and J.  Zhang}, \emph{Splitting multisymplectic integrators for Maxwell's
equations}, J. Comput. Phys., 229 (2010), pp. 4259-4278.


 \bibitem{n22}{\sc J. Hong, L. Ji, and L. Kong}, \emph{Energy-dissipations splitting finite-difference
time-domain method for Maxwell equations with perfectly matched layers}, J. Comput. Phys., 269 (2014), pp.  201-214.

 \bibitem {Leis86}{\sc R. Leis}, \emph{Initial Boundary Value Problems in Mathematical Physics}, Wiley, New York, 1986.

\bibitem{n28} {\sc D. Liang and Q. Yuan}, \emph{The spatial fourth-order energy-conserved S-FDTD scheme for
Maxwell's equations}, J. Comput. Phys., 243 (2013), pp. 344-364.

\bibitem{n29}  {\sc Q. Liu}, \emph{The PSTD algorithm: a time-domain method requiring only two cells
per wavelength}, Microw. Opt. Technol. Lett., 15 (1997), pp.  158-165.


  \bibitem{add5} {\sc J.E. Marsden and A. Weinstein}, \emph{The Hamiltonian structure of the Maxwell-Vlasov
equations}, Physica D, 4 (1982), pp. 394-406.

\bibitem{n33} {\sc P. Monk}, \emph{Finite element methods for Maxwell's equations}, Clarendon press, Oxford,
edition, 2003.

\bibitem{Moya12}{\sc L. Moya}, \emph{Temporal convergence of a locally implicit discontinuous Galerkin method for
Maxwell's equations}, ESAIM Math. Model. Numer. Anal., 46(5):1225-1246, 2012.






      \bibitem {Monk94} {\sc P. Monk and E. S\"{u}li}, \emph{A convergence analysis of Yee's scheme on nonuniform grids},
SIAM J. Numer. Anal., 31 (1994), pp.  393-412.

\bibitem{n35}  {\sc C.D. Munz, P. Ommes,  R. Schneider, E. Sonnendr\"{u}cker, and U. Vo{\ss}}, \emph{Divergence correction techinques for Maxwell solvers based on a hyperbolic model}, J. Comput. Phys., 161 (2000), pp. 484-511.

 \bibitem {4} {\sc T. Namiki}, \emph{A new FDTD algorithm based on alternating direction implicit method}, IEEE
Trans. Micro. Theo. Tech., 47 (1999), pp. 2003-2007.
%

 \bibitem{t13} {\sc T. Pa\v{z}ur},  \emph{Error analysis of implicit and exponential time integration of linear Maxwell's equa-
tions}, PhD thesis, Karlsruhe Institute of Technology, 2013. URL https://publikationen.
bibliothek.kit.edu/1000038617.


\bibitem{n38} {\sc J. Shang}, \emph{High-order compact-difference schemes for time-dependent Maxwell
equations}, J. Comput. Phys., 153 (1999), pp. 312-333.

    \bibitem{Shen}{\sc  J. Shen, T. Tang, and L. Wang}, \emph{Spectral Methods: Algorithms, Analysis, Applications}, Springer, Berlin, 2011.

 \bibitem{n40} {\sc T.W.H. Sheu, Y. Chung, J. Li, and Y. Wang}, \emph{Development of an explicit
non-staggered scheme for solving three-dimensional Maxwell's equations}, Comput.
Phys. Commun., 207 (2016), pp. 258-273.

\bibitem{n41}  {\sc A. Stern, Y. Tong, M. Desbrun, and J.E. Marsden}, \emph{Geometric computational electrodynamics with variational integrators and discrete differential forms}, In: Geometry, Mechanics, and Dynamics, pp. 437-475. Springer, New York, 2015.

 \bibitem{n42}   {\sc H. Su, M. Qin, and R. Scherer}, \emph{A multisymplectic geometry and a multisym-
plectic scheme for Maxwell's equations}, Int. J. Pure. Appl. Math., 34 (2007), pp. 1-17.

\bibitem{n43} {\sc Y. Sun and P.S.P. Tse}, \emph{Symplectic and multi-symplectic numerical methods for
Maxwell's equations}, J. Comput. Phys., 230 (2011), pp. 2076-2094.

\bibitem{n44}  {\sc A. Taflove and S.C. Hagness}, \emph{Computational electrodynamics}, Artech House,
Boston, 2005.

    \bibitem{Trefethen}
{\sc L.N. Trefethen}, \emph{Spectral Methods in MATLAB}, SIAM, Philadelphia, 2000.

\bibitem{Verwer11} {\sc J.G. Verwer}, \emph{Component splitting for semi-discrete Maxwell equations}, BIT, 51 (2011), pp.  427-445.


\bibitem{WZ}
{\sc B. Wang and X. Zhao}, \emph{Error estimates of some splitting schemes for
charged-particle dynamics under strong magnetic field}, SIAM J. Numer. Anal. 59 (2021), pp. 2075-2105.

\bibitem{WW21}
{\sc  X. Wu and B. Wang}, \emph{Geometric integrators for differential equations with highly oscillatory dolutions}, Springer Nature Singapore Pte Ltd. 2021

{
\bibitem{yang}
{\sc  H. Yang, X. Zeng, and X. Wu}, \emph{An approach to solving Maxwell's equations in time domain}, J. Math. Anal. Appl. 518 (2023) 126678.}

 \bibitem {3}  {\sc K.S. Yee}, \emph{Numerical solution of initial boundary value problems involving Maxwell's equations
in isotropic media}, IEEE Trans. Antennas and Propagation, 14 (1966), pp. 302-307.

 \bibitem{n49}  {\sc S. Zhao and G. Wei}, \emph{High-order FDTD methods via derivative matching for
Maxwell's equations with material interfaces}, J. Comput. Phys., 200 (2004), pp. 60-103.

 \bibitem {5}  {\sc F. Zheng, Z. Chen, and J. Zhang}, \emph{Toward the development of a three-dimensional unconditionally stable finite-difference time-domain method}, IEEE Trans. Microwave Theory Tech.,
48 (2000), pp. 1550-1558.

\bibitem{n51} {\sc H. Zhu, S. Song, and Y. Chen}, \emph{Multi-symplectic wavelet collocation method
for Maxwell's equations}, Adv. Appl. Math. Mech., 3 (2011), pp. 663-688.




\end{thebibliography}
\end{document}